\documentclass[12pt,letterpaper,titlepage]{amsart}
\usepackage{amsmath, amssymb, amsthm, amsfonts,amscd,xr}
\input epsf.tex
\newcommand{\N}{\mathbb{N}}

\newcommand{\Z}{\mathbb{Z}}
\newcommand{\R}{\mathbb{R}}
\newcommand{\C}{\mathbb{C}}

\def\beq{\begin{equation}}
\def\eeq{\end{equation}}

\def\arr{\hbox to 20pt{\rightarrowfill}}

\def\e{\varepsilon}

\def\so{\mathfrak {so}}

\def\Sl{\mathrm {Sl}} 
\def\SO{\mathrm {SO}} 
 
\def\D{\mathcal D} 
\def\cC{\mathcal C} 
\def\H{\mathcal H} 
\def\O{\mathcal O} 
 
\def\W{\mathcal W}

\def\U{\mathcal U} 
\def\F{\mathbf F}

\newenvironment{res} 
               {\begin{equation} 
\begin{minipage}{0.85\textwidth}} 
               { \end{minipage}\end{equation} } 

\def\ber{\begin{res} } 
\def\eer{\end{res}} 
 
\numberwithin{equation}{section} 
\newtheorem{thm}{Theorem}[section]

\newcommand{\oline}{\overline}
 
\newtheorem{lemma}[thm]{Lemma} 
\newtheorem{lem}[thm]{Lemma}

\newtheorem{cor}[thm]{Corollary} 
\newtheorem{ex}[thm]{Example} 
\newtheorem{prop}[thm]{Proposition} 
\newtheorem{df}[thm]{Definition} 
 
\newtheorem{rem}[thm]{Remark}

\makeatletter 
\def\section{\@startsection {section}{1}{\z@}{3.5ex plus 1ex minus 
    .2ex}{2.3ex plus .2ex}{\large\bf}} 
    \def\subsection{\@startsection{subsection}{2}{\z@}{3.25ex plus 1ex minus 
 .2ex}{1.5ex plus .2ex}{\bf}} 
\makeatother 
\def\qed{\hfill $\square$\par\vspace{5pt}} 
\def\bysame{\leavevmode\hbox to3em{\hrulefill}\,}

\def\Ad{\operatorname{Ad}}

\def\e{\epsilon} 
\def\g{\gamma}

\def\af{\mathfrak{a}} 
\def\diag{\operatorname {diag}}
\def\Ext{\operatorname{Ext}}
\def\gf{\mathfrak{g}} 
\def\h{\mathfrak{h}}

\def\kf{\mathfrak{k}}

\def\nf{\mathfrak{n}}

\def\pf{\mathfrak{p}} 
 
\def\qf{\mathfrak{q}}

\def\O{\mathcal{O}}

\def\PW{\operatorname{PW}}

\def\H{\mathcal{H}} 
\def\F{\mathcal{F}} 
 
\def\K{\mathcal{K}} 
 
\def\E{\mathcal{E}} 
\def\P{\mathbb{P}} 
\def\S{\mathcal{S}} 
 
\def\W{\mathcal{W}}

\def\Sym{\mathrm{Sym}}

\def\bs{\backslash}

\renewcommand{\Re}{\mbox{\rm Re}\,} 
\renewcommand{\Im}{\mbox{\rm Im}\,} 
\makeindex 

\begin{document} 
\title[Crown theory]
{Crown theory for the upper half plane}

\author{Bernhard Kr\"otz} 
\address{Max-Planck-Institut f\"ur Mathematik\\  
Vivatsgasse 7\\ D-53111 Bonn
\\email: kroetz@mpim-bonn.mpg.de}

\date{\today} 
\thanks{}
\maketitle

\bigskip 
\centerline{\it Stephen Gelbart gewidmet}
\bigskip 
\begin{align*} \hbox{ \it He is Mensch.} & \\ 
& \hbox{\it E. L.}\end{align*}  
\newpage 
\tableofcontents
\newpage
\section{Vorwort} This paper features no introduction; it has 
a table of contents. 
\bigskip \par 

\par The material for this text is scattered throughout 
my work, often only found in unpublished notes of mine.
I focus on the upper half plane but want to mention that most matters 
hold true for arbitrary Riemannian symmetric spaces of the 
non-compact type. When I think it is useful, then remarks and references 
to the more general literature are made.

\par Over the years I  had the opportunity to lecture on the 
crown topic at various institutions; these are: 
\begin{itemize}
\item Research Institute of Mathematical Sciences (R.I.M.S.), 
      Kyoto, various lectures in the fall semester of 2004
\item Indian Statistical Institute, Bangalore, Lectures on the 
      crown domain, March 2005 
\item University of Hokkaido at Sapporo, Center of excellence 
      lecture series "Introduction to complex crowns", May 2005
\item Morningside Center of Mathematics, Academica Sinica, Beijing, 
      "Introduction to complex crowns", lectures for a summer 
      school, July 2005 
\item Max-Planck-Institut f\"ur Mathematik, various presentations. 
\end{itemize}
It is my special pleasure to thank my various hosts at this opportunity 
again.

\newpage

\section{Symbols} Throughout this text capital Latin letters, e.g. $G$,  
will be used for real algebraic groups; $\C$-subscripts will denote
complexifications, e.g. $G_\C$. Lie algebras of groups 
will be denoted by the corresponding lower case altdeutsche 
Frakturschrift, e.g. $\gf$ is the Lie algebra of $G$. 

\par In this paper our concern is with 

$$G=\Sl(2,\R) \quad \hbox{and}\quad G_\C= \Sl(2,\C)\, .$$
The following subgroups of $G$ and their complexifications
will be of relevance for us: 

\begin{align*} A &=\left \{ a_t=\begin{pmatrix} t & 0 \\ 0 & 1/t
\end{pmatrix}\mid t>0\right\}\, ,\\
A_\C &=\left \{ a_z=\begin{pmatrix} z & 0 \\ 0 & 1/z
\end{pmatrix}\mid z\in \C^*\right\}\, ,\end{align*}

$$H=\SO(1,1; \R) \quad \hbox{and} \quad H_\C=\SO(1,1; \C)\, ,$$
$$K=\SO(2,\R) \quad \hbox{and} \quad K_\C=\SO(2, \C)\, ,$$ 
and 
\begin{align*} N &=\left \{ n_x=\begin{pmatrix} 1 & x \\ 0 & 1
\end{pmatrix}\mid x\in\R\right\}\, ,\\
N_\C &=\left \{ n_z=\begin{pmatrix} 1 & z \\ 0 & 1
\end{pmatrix}\mid z\in \C\right\}\, . \end{align*}

\section{The upper half plane, its affine complexification and the crown}
Our concern is with the Riemannian symmetric space 
$$X=G/K$$
of the non-compact type. We usually identify $X$ with the 
upper halfplane ${\bf H}=\{ z\in \C\mid \Im z>0\}$ via the map 

$$X\to {\bf H}, \ \ gK\mapsto {a i +b \over ci +d} \qquad 
\left(g=\begin{pmatrix} a & b\\ c & d \end{pmatrix}\right)\, .$$
We use $x_0=K$ for the base point $eK\in X$ and note that $x_0=i$ 
within our identification. 

\par We view $X={\bf H}$ inside of the complex projective 
space $\P^1(\C)=\C\cup\{\infty\}$ and note that 
$\P^1(\C)$ is homogeneous for $G_\C$ with respect 
to the usual fractional linear action: 

$$g(z)= {a z +b \over cz +d}\qquad \left(z\in \P^1(\C), 
g=\begin{pmatrix} a & b\\ c & d \end{pmatrix}\in G_\C\right)\, .$$  

\par Upon complexifiying $G$ and $K$ we obtain 
the {\it affine complexification} 
$$X_\C =G_\C/K_\C$$
of $X$. Observe that the map 
\begin{equation}\label{e1}  X\hookrightarrow X_\C, \ \ gK\mapsto gK_\C\end{equation}
constitutes a $G$-equivariant embedding which realizes 
$X$ as a totally real submanifold of $X_\C$. 
We will use a more concrete model for  $X_\C$: 
the mapping 
$$X_\C\to \P^1(\C)\times \P^1(\C)\bs{\diag}, \ \ gK_\C\mapsto \left(g(i), g(-i)\right)  $$
is a $G_\C$-equivariant diffeomorphism. With this 
identification of $X_\C$ the embedding of (\ref{e1}) becomes 
\begin{equation}\label{e2}  X\hookrightarrow X_\C, \ \ z\mapsto (z,\oline z) \, .\end{equation}

\par We will denote by $\oline X$ the lower half plane and arrive 
at the object of our desire:

$$\Xi= X\times \oline X$$
the {\it crown domain} for $\Sl(2,\R)$. 
Let us list some obvious properties of $\Xi$ and emphasize that they 
hold for arbitrary crowns: 
\begin{itemize}
\item $\Xi$ is a $G$-invariant Stein domain in $X_\C$. 
\item $G$ acts properly on $\Xi$. 
\item $\Xi=X\times \oline X$ is the {\it complex double} -- this always holds if 
      the underlying Riemannian space $X=G/K$ is already complex.  
\end{itemize}
 
\section{Geometric structure theory}
\subsection{Basic structure theory}
\subsubsection{$\Xi$ as a union of elliptic $G$-orbits}
We note that 
$$\af=\left\{ \begin{pmatrix} x & 0 \\ 0 & -x\end{pmatrix}\mid x\in\R
\right\}$$
and focus on a domain inside: 
$$\Omega=\left\{ \begin{pmatrix} x & 0 \\ 0 & -x\end{pmatrix}\mid 
x\in(-\pi/4, \pi/4)\right\}\, .$$
We note that $\Omega$ is invariant under the Weyl group 
$\W=N_K(A)/Z_K(A)\simeq \Z_2$ and that 
that $\exp(i\Omega)$ consist of elliptic elements 
in $G_\C$. 
\par The following proposition constitutes of what we call the 
{\it elliptic parameterization} of the crown domain. 

\begin{prop}\label{p=e} $\Xi=G\exp(i\Omega)\cdot x_0$. 
\end{prop}
\begin{proof} (cf.\ \cite{KSII}, Th. 7.5 for the most general case). 
We first show that $G\exp(i\Omega)\cdot x_0\subset \Xi$. 
By $G$-invariance of $\Xi$, this reduces to verify that 
$$\exp(i\Omega).x_0\in\Xi\, .$$
Explicitly this means 
$$(e^{2i\phi} i,  -e^{2i\phi }i)\in X\times \oline X$$
for $\phi\in (-\pi/4, \pi/4)$;  evidently 
true. 
\par Conversely, we want to see that every element 
in $\Xi$ lies on a $G$-orbit through $\exp(i\Omega)$. 
Let $S=G\times G$ and $U=K\times K$ and observe, that 
$\Xi=S/U$ as homogeneous space. Now 
$$S=\diag (G) \operatorname{antidiag}(H) U$$
and all what we have to see is that 
$$ \operatorname{antidiag}(H) \cdot x_0\subset G\exp(i\Omega)\cdot x_0, $$
or, more concretely, 

\begin{equation}\label{i1} \{ \left( {i \cosh t  +\sinh t\over i\sinh t  +\cosh t}, 
- {i\cosh t  +\sinh t\over i\sinh t  +\cosh t}\right)\mid 
t\in\R\} \subset  G\exp(i\Omega)\cdot x_0\, .\end{equation} 
Now we use that 
$A\exp(i\Omega) (i) = X$
and conclude that the LHS of (\ref{i1}) is contained 
in  $A\exp(i\Omega)\cdot x_0$. 
\end{proof}

\subsubsection{$\Xi$ as a union of unipotent $G$-orbits}
The following parameterization of $\Xi$ is relevant 
for our discussion of automorphic cusp forms at the end of this article. 
It was discovered in \cite{KO}.

We consider the Lie algebra of $N$:
$$\nf=\left\{\begin{pmatrix} 0 & x \\ 0 & 0\end{pmatrix}\mid x\in\R\right\}$$
and focus on the subdomain
$$\Lambda=\left\{\begin{pmatrix} 0 & x \\ 0 & 0\end{pmatrix}\mid x\in (-1,1)
\right\}\, .$$
The following proposition constitutes of what we call the 
{\it unipotent parameterization} of the crown domain, see
 \cite{KO}, Th. 3.4 for $G=\Sl(2,\R)$ and \cite{KO}, Th. 8.3 for 
$G$ general.

\begin{prop}\label{p=u} $\Xi=G\exp(i\Lambda)\cdot x_0$. 
\end{prop}
\begin{proof} We wish to give the more conceptual proof. 
Let us first see that $G\exp(i\Lambda)\cdot x_0\subset \Xi$, i.e.
$$\exp(i\Lambda)\cdot \subset \Xi\, .$$ 
Concretely this means that 
$$(ix + i, - i +ix )\in X\times \oline X$$
for all $x\in (-1,1)$; evidently true. 

\par For the reverse inclusion we will borrow in content and notation from 
Subsubsection \ref{s=c} from below. It is a conceptional argument. 
Fix $Y\in\Omega$. Then, according to the complex convexity 
theorem \ref{t=cc} there exist a $k\in K$ such that 
$$\Im \log a_\C(k\exp(i Y)\cdot x_0)=0\, .$$
In other words, 
$$k\exp(iY)\cdot x_0\in N_\C A\cdot x_0= A N_\C \cdot x_0\, .$$
We conclude that $\exp(iY)\cdot G \exp(i\nf)\cdot x_0$. 
From our discussion in (i) we deduce that 
$\exp(iY)\cdot x_0\in G\exp(i\Lambda)\cdot x_0$. 
\end{proof}

Another way to prove Prop. \ref{p=u} is by means of 
matching elliptic and unipotent $G$-orbits. We 
cite \cite{KO}, Lemma 3.3: 

\begin{lemma} \label{l=match} For all $\phi\in (-\pi/4,\pi/4)$ the following identity 
holds: 
$$G\begin{pmatrix} 1 & i \sin 2\phi\\ 0 & 1\end{pmatrix}\cdot x_0= 
G\begin{pmatrix} e^{i\phi} & 0\\ 0 & e^{-i\phi}\end{pmatrix}\cdot x_0\, .$$
\end{lemma}
\begin{proof} This is best seen in the hyperbolic model 
of the crown which we discuss in Appendix A; the proof
of the lemma will be given  there, too. 
\end{proof} 

\subsubsection{Realization in the tangent bundle}
Let 

$$\pf=\Sym(2,\R)_{\rm {tr}=0}$$
and recall that: 
\begin{itemize} 
\item $\gf=\kf\oplus \pf$, the Cartan decomposition; 
\item $\pf$ is a linear $K$-module which naturally identifies 
      with $T_{x_0} X$, the tangent space of $X$ at $x_0$. 
\end{itemize}
We write $TX$ for the tangent bundle which is naturally isomorphic 
with $G\times_K \pf$ via the map 
$$G\times_K\pf \to T X, \ \ [g,Y]\mapsto {d\over dt}\Big|_{t=0} g\exp(tY)\cdot x_0\, .$$

Inside $\pf$ we consider the disc 

$$\hat\Omega=\{ Y\in \pf\mid \operatorname{spec} (Y)\subset (-\pi/4, \pi/4)\}$$
and note that $\hat \Omega$ is $K$-invariant and 
$$\hat\Omega\cap \af =\Omega\, .$$
Therefore we can form the disc-bundle $G\times_K \hat\Omega$ inside of 
$TX$. 

The following result was obtained in \cite{AG}, in full 
generality. 

\begin{prop}\label{p=ag} The map 
$$ G\times_K \hat\Omega\to \Xi, \ \ [g,Y]\mapsto g\exp(iY)\cdot x_0$$
is a $G$-equivariant diffeomorphism. 
\end{prop}

\begin{proof} Ontoness is clear. Injectivity can be obtained 
by direct computation. 
\end{proof}

\begin{rem} The above proposition becomes more interesting 
when one considers more general groups $G$ -- the statement 
is literally the same. One deduces that $G$ acts properly 
on $\Xi$ (the action of $G$ on $TX$ is proper) and that 
$\Xi$ is contractible: $\Xi$ is a fiber bundle over 
$X=G/K\simeq \pf$ with convex fiber $\hat \Omega$. 
\end{rem} 

\subsubsection{The various boundaries of the crown} In this part 
we discuss the various boundaries of $\Xi$. First and foremost 
there is the topological boundary $\partial\Xi$ of $\Xi$ in 
$X_\C$. We will see that $\partial\Xi$ carries a natural 
structure of a cone bundle over the affine symmetric space $Y=G/H$. 
In particular $Y\subset \partial\Xi$ and $Y$ and we will show 
that $Y$ is some sort 
of Shilov boundary of $\Xi$ ( we will call it the 
{\it distinguished boundary} though).

\par We write $\qf$ for the tangent space of $Y$ at the base point 
$y_0=H\in Y$. Note that 
$$\qf = \R \underbrace{\begin{pmatrix} 1 & 1\\ -1 & -1\end{pmatrix}}
_{:={\bf e}}
\oplus 
\R \underbrace{\begin{pmatrix} 1 & -1\\ 1 & -1\end{pmatrix}}_{:={\bf f}}$$
is the decomposition of the $H$-module in eigenspaces. 
In particular, 
$$C:=\R_{\geq 0}{\bf e}\cup \R_{\geq 0}{\bf f}$$
is an $H$-invariant cone in $\qf$ and we can form the cone bundle 
$$\cC:=G\times_H C$$
inside of $TY$. 
\par We note that $Y$ is naturally realized in $X_\C$ via 
the map 
$$Y\to X_\C, \ \ gH\mapsto g(1, -1)\, , $$ 
i.e. $y_0$ identifies with $(1,-1)$. 

\begin{prop}\label{p=b}
$$\cC=G\times_H C \to \partial \Xi, \ \ [g,Z]\mapsto
g\exp(iZ)\cdot y_0$$
is a $G$-equivariant homeomorphism. 
\end{prop}
\begin{proof} Direct computation; see \cite{KO}, Th. 3.1 for 
details. 
\end{proof}

\begin{cor} $\pi_1(\partial\Xi)=\pi_1(G/H)=\Z\, .$
\end{cor}

Henceforth we call write $\partial_d\Xi =G\cdot y_0\simeq Y$ 
and call $\partial_d\Xi$ the distinguished boundary 
of $\Xi$. Its relevance is as follows. Write
$\P(\Xi)$ for the cone of strictly plurisubharmonic functions
on $\Xi$ which extend continuously up to the boundary. 
A simple exercise in one complex variable then yields 
(cf.\ cite{GKI},  Th. 2.3). 

\begin{lem} For all $f\in \P(\Xi)$: 
$$\sup_{z\in \Xi} |f(z)|= \sup_{z\in \partial_d\Xi} |f(z)|\, .$$
\end{lem} 

The complement of the distinguished boundary of $\Xi$ we denote 
$\partial_u\Xi$, and refer to it as the unipotent boundary. 
A straightforward computation explains the
terminology: 

\begin{equation} \label{b=u} \partial_u\Xi =G\begin{pmatrix} 1 & i \\ 0 & 1\end{pmatrix}\cdot x_0 \amalg 
G\begin{pmatrix} 1 & -i \\ 0 & 1\end{pmatrix}\cdot x_0\, .\end{equation}

\subsection{Fine structure theory}

\subsubsection{The complex convexity theorem}\label{s=c}
We begin the  standard horospherical coordinates for $X$: 
the map 

$$ N\times A \to X, \ \ (n_x, a_{\sqrt y})\mapsto n_x a_{\sqrt{y}}\cdot i = x+iy$$
is an analytic diffeomorphism. Accordingly 
we obtain a map $a: X\to A$, the so-called  $A$-projection. 
Upon complexifying $X=NA\cdot x_0$ we obtain a Zariski-open 
subset 

$$ N_\C A_\C \cdot x_0\subsetneq X_\C\, .$$
Upon extending the map $a$ holomorphically we have to be more 
careful as the groups $A_\C$ and $K_\C$ intersect 
in the finite two-group 
$$M=A_\C \cap K_\C =\{ \pm {\bf 1}\}\, .$$
Accordingly the extension $a_\C$ is only 
valued mod $M$: 

$$a_\C : N_\C A_\C \cdot x_0\to A_\C/M\, .$$
The second part of the following proposition 
is of fundamental importance. 

\begin{prop}\label{p=h} The following assertions hold: 
\begin{enumerate}
\item $N_\C A_\C\cdot x_0=\C\times \C \bs\diag$, in other 
words $N_\C A_\C\cdot x_0$ is the affine open piece 
of $X_\C$. 
\item $\Xi\subset N_\C A_\C\cdot x_0$. 
\item The map $a_\C$, restricted to $\Xi$, admits 
a holomorphic logarithm $\log a_\C: \Xi \to \af_\C $ such that 
$\log a_\C (x_0)= 0$. 
\end{enumerate}
\end{prop}

\begin{proof} (i) We observe that 
\begin{align*} N_\C A_\C \cdot x_0 &=\{ (iz +w, -iz +w)\mid z\in \C^*, w\in \C\} \\
&= \{ (z +w, -z +w)\mid z\in \C^*, w\in \C\} \\
&=\C\times \C \bs\diag\, .\end{align*}
\par\noindent (ii) is immediate from (i). 
\par\noindent (iii) follows from (ii) and the fact that $\Xi$ 
is simply connected. 
\end{proof}

\begin{rem} We wish to make a few remarks about the inclusion (ii) 
for more general groups. For classical groups (ii) was 
obtained in \cite{KSI} and \cite{GM} by somewhat explicit, although  
efficient, matrix computations. For general simple groups 
a good argument based on complex analysis was given in 
\cite{H1} and \cite{H2}. The method of \cite{H1} was later 
simplified and slightly generalized in \cite{M}. 
\end{rem}

From Proposition \ref{p=h}(i) we obtain the 
following

\begin{cor} $\left[\bigcap_{g\in G} g N_\C A_\C \cdot x_0\right]_0=
\Xi$, where $[\cdot]_0$ denotes the connected component of 
$[\cdot]$ containing $x_0$. 
\end{cor}

\begin{proof} Let $D:=\left[\bigcap_{g\in G} g N_\C A_\C \cdot x_0
\right]_0$. Write $D_1$, $D_2$ for the projection 
of $D$ to the first, resp. second,  factor in $[\C\times \C]\bs 
\diag$. Then $D_1\subset\C$ is $G$-invariant. 
Hence $D_1=X$, $D_1=\oline X$ or $D_1=X\cup \oline X$. 
The last case is excluded, as $D$ is connexted. The second 
case is excluded as $x_0\in D$ implies $i\in D_1$. 
Hence $D_1=X$. By the same reasoning 
one gets $D_2=\oline X$. As $\Xi\subset D$  we thus 
get $D=\Xi$.
\end{proof}

For an element $Y\in \af$ we note that the convex hull 
of the Weyl-group orbit of $Y$, in symbols 
$\operatorname{conv} (\W\cdot Y)$, is just the line segment $[-Y,Y]$. 
With that we turn to a deep geometric fact for crown domains, the 
{\it complex convexity theorem}: 

\begin{thm}\label{t=cc} For $Y\in \Omega$: 
$$\Im \log a_\C (K\exp(iY)\cdot x_0)= [-Y,Y]\, .$$
\end{thm} 
  
\begin{proof} Direct computation. For $G=\Sl(2,\R)$ there is an 
explicit formula for $a_\C$: with  
$k_\theta=\begin{pmatrix} \cos\theta & \sin \theta\\ -\sin\theta& \cos \theta
\end{pmatrix}\in K$ one has 
$$a_\C \left(k_\theta \begin{pmatrix} e^{i\phi} & 0 \\ 0 & e^{-i\phi} \end{pmatrix}\right)
= a_z$$
with 
$$z = \sqrt{ e^{2i\phi} +\sin^2 \theta (e^{-2i\phi} + e^{2i\phi})}\, , $$
see \cite{KSI}, Prop. A.1 (i). From that the assertion follows. 
For the general case we refer to \cite{GKII} for the inclusion "$\subset$" and 
to \cite{KOt} for actual equality.  
\end{proof}

\subsubsection{Realization in the complexified 
Cartan decomposition}\label{s=cd}
The Cartan or polar decomposition of $X$ says that the map 
$$K/M\times A\to X, \ \ (kM,a)\mapsto ka\cdot x_0$$
is onto with faithful restriction to $K/M\times A^+$. Here, as usual 
$$A^+=\{ a_t \mid t>1\}\, .$$  
Thus 
$$X=KA\cdot x_0$$
and we wish to complexify this equality. We have to be a little more 
careful here, as $K_\C A_\C \cdot x_0$ is no longer a domain (it fails
to be open at the base point $x_0$). 
The remedy comes from a little bit of invariant theory. 
We note that   $X_\C$ is an affine variety and write 
$\C[X_\C]$ for its ring of regular function. We denote by 
$\C[X_\C]^{K_\C}$ for the subring of regular 
function. According to Hilbert, the invariant ring 
is finitely generated, i.e.  
$$\C[X_\C]^{K_\C}=\C[p]\, .$$
In order to describe $p$ we use a different 
realization of $X_\C$, namely 
$$X_\C=\Sym(2,\C)_{\det=1}\, .$$
In this model the generator $p$ is given by 
$$p: X_\C \to \C, \ \ z\mapsto \operatorname{tr} z\, .$$
For a symmetric, i.e. $\W$-invariant,  open segment 
$\omega \subset \Omega$ we define a $K_\C$-invariant 
domain $X_\C(\omega)\subset X_\C$ by 

$$X_\C(\omega)= p^{-1}(p(A\exp(i\omega)\cdot x_0))\, .$$
We note that 
\begin{itemize}
\item $K_\C A\exp(i\omega)\cdot x_0\subset X_\C(\omega)$
\item $\exp(i\omega')\cdot x_0\not\subset X_\C (\omega)$ if 
$\omega\subsetneq \omega'$. 
\end{itemize}
Hence we may view $X_\C(\omega)$ as the $K_\C$-invariant 
open envelope of $K_\C A\exp(i\omega)\cdot x_0$ in $X_\C$. 
The main result here is as follows: 

\begin{thm} \label{th=coca}For all open symmetric segments 
$\omega\subset\Omega$ one has 
$$G\exp(i\omega)\cdot x_0\subset X_\C (\omega)\, .$$
In particular 
$$\Xi\subset X_\C(\Omega)\, .$$
\end{thm}

\begin{proof} For $G=\Sl(2,\R)$ this was established in 
\cite{KOt2}; in general in \cite{K1}. 
\end{proof}

\section{Holomorphic extension of representations}

I want to explain a few things 
on representations first.
For the beginning $G$ might be any connected unimodular Lie group, 
for simplicity even contained in its universal complexification $G_\C$. 
By a {\it unitary representation} of $G$ we understand 
a group homomorphism
$$\pi: G\to U(\H)$$
from $G$ into the unitary group of some complex Hilbert
space $\H$ such that for all $v\in \H$ the orbit maps
$$f_v: G\to \H, \ \ g\mapsto \pi(g)v$$
are continuous.   
We call a vector $v\in \H$ {\it analytic} if 
$f_v$ is a real analytic $\h$-valued map. The entity 
of all analytic vectors of $\pi$ is denoted by 
$\H^\omega$ and we observe that $\H^\omega$ is a 
$G$-invariant vector space. 
The following result was obtained by Nelson; the idea is 
already found in the approximation theorem of 
Weierstra\ss. 

\begin{lem} $\H^\omega$ is dense in $\H$. 
\end{lem}

\begin{proof} (Sketch) We first recall that 
with $\pi$ comes a Banach-$*$-representation $\Pi$ of 
the group algebra $L^1(G)$ given by 
$$\Pi(f)v =\int_G f(g) \pi(g) v \ dg \qquad (f\in L^1(G), v\in\H) $$
with $dg$ a  Haar-measure.
For a Dirac-sequence $(f_n)_{n\in\N}$ in $L^1(G)$ one 
immediately verifies that 
\begin{equation}\label{e=d} \Pi(f_n)v \to v \end{equation}
for all $v\in\H$. 
We choose a good Dirac sequence: Fix a left invariant 
Laplace operator on $G$ and write $\rho_t$ for the corresponding 
heat kernel. We use the theory of parabolic PDE's as black box 
and just state: 

\begin{itemize} \item $\rho_t\in L^1(G)$ for all $t>0$,  
\item $\rho_t$ is analytic and of Gau\ss{}ian decay, 
\item $(\rho_{1/n})_{n\in\N}$ is a Dirac-sequence.    
\end{itemize}
As a result $\Pi(\rho_t)v \in \H^\omega$ and 
$$\lim_{t\to 0^+} \Pi(\rho_t)v = v\qquad (v\in\H)$$
by (\ref{e=d}). 
\end{proof}

Let us now sharpen the assumptions on $G$ and $\pi$. In the next step 
we request: 

\begin{itemize} 
\item $G$ is semisimple. 
\item $\pi$ is irreducible. 
\end{itemize}

Harish-Chandra observed that screening the representation $\pi$ under 
a maximal compact subgroup $K<G$ is meaningful. 
He introduced 
the space of $K$-finite vectors:
$$\H_K=\{ v\in \H\mid \operatorname{span}_\C \{ \pi(K)v\}\  \hbox {is finite dim.}\}\, $$

Observe that $\H_K$ is dense in $\H$ by the theorem of Peter and Weyl. 
Harish-Chandra made a key-observation: 

\begin{lem} $\H_K\subset\H^\omega$. 
\end{lem}
\begin{proof} The following sketch of proof is non-standard. 
We will use a little bit of functional analysis. It is 
known that $\H^\omega$ is a locally convex vector space 
of compact type. As such it is sequentially complete. This makes 
the Peter-Weyl-Theorem for the representation of $K$ on 
$\H^\omega$ applicable. In particular the $K$-finite 
vectors in $\H_K^\omega$ in $\H^\omega$ are dense in $\H^\omega$.
Apply the previous Lemma combined with the density of 
$\H_K$ in $\H$.  
\end{proof}

The upshot of our discussion is that $\H_K$ is the vector space 
consisting of the best possible analytic vectors. It is a module 
of countable dimension for the Lie algebra $\gf$ and as such irreducible.

\par Given $v\in \H_K$ we consider the real analytic orbit map 
$$f_v: G\to \H, \ \ g\mapsto\pi(g)v$$  
and ask the following : 

\bigskip\noindent {\it Question: What is the natural domain 
$D_v\subset G_\C $ to which $f_v$ extends holomorphically?}

\bigskip It turns out that $D_v$ does only depend on the 
type of the representation $\pi$ but not on the specific 
vector $v\neq 0$ (this is reasonable as $v$ generates $\H_K$ as 
a $\gf$-module). We will give this classification 
in the subsection below. At this point we only remark that 
the domain $D_v$ is naturally left $G$-invariant and right $K_\C$-invariant, 
in symbols: 
$$D_v= GD_vK_\C\, .$$
 
\par A little bit more terminology is good for the purpose 
of the discussion. We write 
$$q: G_\C \to X_\C, \ \ g\mapsto gK_\C$$
for the canonical projection and for a domain $D\subset X_\C$ we write 
$$DK_\C =q^{-1}(D)$$
for the pre-image of $D$ in $G_\C$.

To get a feeling for that I want to discuss
one class of examples first. 

\subsection{The spherical principal series} For the rest of 
this section we return to our basic setup: $G=\Sl(2,\R)$. 

\par We fix a parameter $\lambda\in \R$, let $\H=L^2(\R)$ and declare 
a unitary representation $\pi_\lambda$ of $G$ on $\H$ via 
\begin{equation} \label{(6.1)} [\pi_\lambda(g)f] (x) = |cx+d|^{-1 +i\lambda} 
f\left( {ax+b\over cx +d}\right) \end{equation}  
for $g^{-1}=\begin{pmatrix} a & b\\ c & d \end{pmatrix}$, $f\in \H$ and 
$x\in\R$.     
In the literature one finds $\pi_\lambda$ under the term 
{\it spherical unitary  principal series}. 
This representation is {\it $K$-spherical}, i.e. the space of $K$-fixed 
vectors $\H^K$ is non-zero. 
More precisely, $\H^K =\C v_K$ with 
$$v_K(x)={1\over \sqrt{\pi}}\cdot  {1\over (1 +x^2)^{{1\over 2}(1 -i\lambda)}}$$
being a normalized representative. 
With $v_K$ we form the matrix coefficient 

$$\phi_\lambda(g):=\langle \pi_\lambda(g) v_K, v_K\rangle \qquad (g\in G)\, .$$
The function $\phi_\lambda$ is $K$-invariant from both sides, in particular 
descends to an analytic function on $X=G/K$, also denoted by $\phi_\lambda$. 
We record the integral representation for $\phi_\lambda$: 

$$\phi_\lambda(x)=\int_K a(kx)^{\rho(1 + i\lambda)} \ dk \qquad (x\in X)$$
where $dk$ is a normalized Haar measure on $X$, and the other 
notation standard too: for $\mu\in\af_\C^*$ and $a\in A$ we let 
$a^\mu:=e^{\mu(\log a)}$ and $\rho\in \af^*$ is fixed by 
$\rho\begin{pmatrix} 1 & 0 \\ 0 & -1\end{pmatrix}=1$.  
Now in view of Proposition \ref{p=h}(iii), this implies that 
$\phi_\lambda$ extends to a holomorphic function on $\Xi$ 
given by 
$$\phi_\lambda(z)= \int_K a_\C(kz)^{\rho(1 + i\lambda)} \ dk \qquad (z\in \Xi)\, .$$
With a little bit of functional analysis one then gets that 
the orbit map $f_{v_K}$ extends holomorphically to $\Xi K_\C$. 
Since $\H_K={\mathcal U}(\gf_\C)v_K$ we thus deduce 
that $f_v$ extends to $\Xi K_\C$ for all $v\in \H_K$. 
For $v\neq 0$, this is actually a maximal domain, but that would 
require more work. We summarize the discussion: 

\begin{prop} Let $\pi_\lambda$ be a unitary spherical principal 
series, then for all $v\in \H_K$, the orbit map 
$f_v: G\to \H$ extends to a holomorphic function on $\Xi K_\C$. 
\end{prop}

\begin{rem} Observe that the above proposition implies that 
$\phi_\lambda$ extends holomorphically to $\Xi$. 
\end{rem}

\subsection{A complex geometric classification of $\hat G$}

\subsubsection{More geometry} Before we turn to the subject proper 
we have to introduce two more geometric objects. 
We define two $G$-invariant domains in $X_\C$ by 

\begin{align*}\Xi^+&= X \times \P^1(\C)\bs \diag\, ,\\
\Xi^-&= \P^1(\C) \times \oline X \bs \diag\, .
\end{align*}

We immediately observe that both $\Xi^+$ and $\Xi^-$ feature 
the following properties: 

\begin{itemize} 
\item $G$ acts properly on $\Xi^+$ and $\Xi^-$, 
\item Both $\Xi^+$ and $\Xi^-$ are maximal $G$-domains
in $X_\C$ with proper actions,  
\item Both $\Xi^+$ and $\Xi^-$ are Stein,  
\item $\Xi^+\cap \Xi^-=\Xi$. 
\end{itemize}

In terms of structure theory one can define $\Xi^+$ and $\Xi^-$ as 
follows. Let us denote by $Q^\pm $ the stabilizer of $\pm i $ in $G_\C$. 
Note that $Q^\pm= K_\C \rtimes P^\pm $ with 
$$P^\pm=\left \{ \begin{pmatrix} 1 + z & \mp i z\\ \mp iz& 1 -z\end{pmatrix}  \mid z\in\C\right\}\, .$$
We easily obtain: 
\begin{lem} The following assertions hold: 
\begin{enumerate} 
\item $\Xi^+ K_\C= G K_\C P^+$, 
\item $\Xi^- K_\C = G K_\C P^-$.
\end{enumerate}
\end{lem}

\subsubsection{The classification theorem}

In this section $(\pi,\H)$ denotes an irreducible unitary representation of $G$. 
We call $\pi$ a {\it highest weight}, resp. {\it lowest weight}, representation if 
$\operatorname{Lie}(P^+)$, resp. $\operatorname{Lie}(P^-)$,  acts finitely
on $\H_K$. 
We state the main result (cf. \cite{KO} for $\Sl(2,\R)$ and \cite{K2}
in general).

\begin{thm} \label{t=m}Let $(\pi, \H)$ be a unitary irreducible representation of 
$G$. Let $0\neq v\in \H_K$ be a $K$-finite vector. 
Then a maximal $G\times K_\C$-invariant domain $D_v$ 
to which 
$$f_v: G\to \H, \ \ g\mapsto \pi(g)v$$
extends as a holomorphic function is given as follows: 
\begin{enumerate} 
\item $G_\C$, if $\pi$ is the trivial representation; 
\item $\Xi^+ K_\C$,  if $\pi$ is a non-trivial highest weight representation; 
\item $\Xi^- K_\C$, if $\pi$ is a non-trivial lowest weight representation;   
\item $\Xi K_\C$ in all other cases. 
\end{enumerate}
\end{thm}

It is our desire to explain how to prove this theorem. 
We found out that there is an intimate relation of this 
theorem with proper actions of $G$ on $X_\C$. 

\subsubsection{Proper actions and representations}
The material in this section is taken from \cite{KO}, Section 4. It 
holds for a general semisimple group. 
We begin with a simple reformulation of the Riemann-Lebesgue 
Lemma for representations. 

\begin{lemma} \label{lem=hm}Let $(\pi, \H)$ be a unitary representation
of $G$ which does not contain
the trivial representation. Then $G$ acts properly
on $\H -\{0\}$.
\end{lemma}

\begin{proof} Let $C\subset \H -\{0\}$ be a compact subset
and $C_G=\{ g\in G\mid \pi(g)C\cap C\neq \emptyset\}$.
Suppose that $C_G$ is not compact. Then there exists
a sequence $(g_n)_{n\in \N}$ in $C_G$ and a
sequence $(v_n)_{n\in \N}$ in $C$ such that
$\pi(g_n)v_n \in C$ and $\lim_{n\to \infty} g_n =\infty$.
As $C$ is compact we may assume that $\lim_{n\to\infty} v_n =v$ and
$\lim_{n\to\infty} \pi(g_n)v_n =w$ with $v,w\in C$.
We claim that
\begin{equation} \label{hm}
\lim_{n\to\infty} \langle \pi(g_n)v, w\rangle \neq 0\, . \end{equation}

In fact $\|\pi(g_n) v_n -\pi(g_n)v\|=\| v_n -v\|\to 0$
and thus $\pi(g_n)v\to w$ as well. As $w\in C$, it follows
that $w\neq 0$ and our claim is established.
\par Finally we observe that (\ref{hm}) contradicts
the Riemann-Lebesgue lemma for representations which asserts
that the matrix coefficient vanishes at infinity.
\end{proof}

{}From Lemma \ref{lem=hm} we deduce the following
result.

\begin{thm}\label{t1} Let $(\pi, \H)$ be an irreducible unitary
representation of $G$ which is not trivial. Let $v\in \H_K$,
$v\neq 0$,
be a $K$-finite vector. Let $\tilde D$ be a 
$G\times K_\C$-invariant domain in $G_\C$ with respect to
the property that the
orbit map $F_v: G\to \H, \ \ g\mapsto \pi(g)v$
extends to a $G$-equivariant holomorphic map
$\tilde \Xi\to \H$. Then $G$ acts properly on
$\tilde D/ K_\C \subset X_\C$.
\end{thm}

\begin{proof} We argue by contradiction and assume that
$G$ does not act properly on $D=\tilde D/K_\C$.
We obtain sequences $(z_n')_{n\in \N}\subset D$ and
$(g_n)_{n\in \N} \subset G$ such that
$\lim_{n\to\infty} z_n' =z'\in D$, $\lim_{n\to\infty} g_n z_n' =w'\in D$ and
$\lim_{n\to\infty} g_n=\infty$. We select preimages $z_n$, $z$ and $w$ of
$z_n'$, $z'$ and $w'$ in $\tilde D$. We may assume
that $\lim_{n\to\infty} z_n=z$ and find a sequence $(k_n)_{n\in \N}$ in $K_\C$
such that $\lim_{n\to\infty} g_n z_n k_n = w$.
\par Before we continue we claim that
\begin{equation}\label{eq=nz} (\forall z\in \tilde D)\qquad \pi(z)v\neq 0
\end{equation}
In fact assume $\pi(z)v=0$ for some $z\in \tilde D$. Then
$\pi(g)\pi(z)v=0$ for all $g\in G$. In particular
the map $G\to \H, \ \ g\mapsto \pi(g)v$ is constantly
zero. However this map extends to a holomorphic map
to a $G$-invariant neighborhood in $G_\C$. By the
identity theorem for holomorphic functions this
map has to be zero as well. We obtain a contradiction
to $v\neq 0$ and our claim is established.

\par Write $V= {\rm span} \{\pi(K)v\}$ for the finite dimensional
space spanned by the $K$-translates of $v$.
In our next step we claim that

\begin{equation}\label{eq=bou} (\exists c_1, c_2>0) \qquad  c_1 < \|\pi(k_n)v\|< c_2\, .
\end{equation}
In fact from
$$\lim_{n\to\infty} \pi(g_nz_nk_n)v = \pi(w)v\quad \hbox{and}\quad
\|\pi(g_n z_n k_n)v\|=\|\pi (z_n) \pi(k_n)v\|$$
we conclude with
(\ref{eq=nz}) that
there are positive constants $c_1',c_2'>0$ such that
$c_1'<\|\pi (z_n) \pi(k_n)v\|<c_2'$ for all $n$.
We use that $\lim_{n\to\infty} z_n =z\in \tilde D$ to obtain
$\pi(z_n)|_V-\pi(z)|_V\to 0$ and our claim
follows.

\par We define $C$ to be the closure of the sequences
$(\pi(z_nk_n)v)_{n\in \N}$ and $(\pi(g_nz_nk_n)v)_{n\in \N}$
in $\H$.
With our previous claims (\ref{eq=nz}) and (\ref{eq=bou}) we obtain
that $C\subset \H -\{0\}$ is a compact subset. But
$C_G=\{g\in G\mid \pi(g) C\cap C\neq \emptyset\}$ contains
the unbounded sequence $(g_n)_{n\in \N}$ and hence is
not compact - a contradiction to Lemma \ref{lem=hm}.
\end{proof}

\subsubsection{Remarks on the proof of Theorem \ref{t=m}}
We are going to discuss the various cases in the Theorem. 
\par \bigskip  {\it Case 1: $\pi$ is trivial.} This is clear.
\par \bigskip  {\it Case 2: $\pi$ is a non-trivial highest weight representation.} 
In this case all orbit maps 
$f_v: G\to \H$ of $K$-finite vectors $v$ extend to 
$GK_\C P^+$. As $GK_\C P^+/K_\C =\Xi^+$ and $\Xi^+\subset X_\C $ is maximal 
for proper $G$-action, the assertion follows from 
Theorem \ref{t1}. 
\par \bigskip {\it Case 3: $\pi$ is a non-trivial lowest weight representation.}
Argue as in case 2. 
\par \bigskip {\it Case 4: The remaining cases. } Here we restrict ourselves 
to spherical principal series $\pi_\lambda$. 
We have already seen that $D_v \supset \Xi K_\C$. The remaining 
inclusion will follow from the following Theorem, cf. \cite{GKO} Th. 5.1. 

\begin{thm}\label {t=he} The crown is a maximal $G$-invariant domain
on $X_\C$ to which a spherical function $\phi_\lambda$, $\lambda\in \R$,
extends holomorphically. 
\end{thm}

In order to prove  this result we need some preparation first.
We recall the domain $X_\C(\Omega)$ from Subsection \ref{s=cd}. 
Likewise one defines
$$X_\C(2\Omega)= p^{-1}p(A\exp(2i\Omega)\cdot x_0)\, .$$
Here is the first Lemma. 

\begin{lem} $\phi_\lambda$ extends to a $K_\C$-invariant holomorphic function 
on $X_\C(2\Omega)$. 
\end{lem}
\begin{proof} Recall that $\phi_\lambda$ can be written as 
a matrix coefficient 
$$\phi_\lambda(x)=\langle \pi_\lambda(x) v_K, v_K\rangle\, .$$ 
For $x= a\exp(2iY)\cdot x_0$ with $a\in A$ and $Y\in\Omega$ we now set 
\begin{equation} \label{s=pos} \phi_\lambda(a\exp(2iY)\cdot x_0)= 
\langle \pi_\lambda(a\exp(iY)) v_K, \pi_\lambda(\exp(iY)v_K\rangle\, .\end{equation}
It is easy to see that this is well defined and holomorphic 
on $A\exp(2i\Omega)\cdot x_0$. Extend by $K_\C$-invariance. 
\end{proof}

\begin{rem} We will show below that $X_\C(2\Omega)$ is the largest $K_\C$-domain 
to which $\phi_\lambda$ extends holomorphically. 
\end{rem}

Explicitly the $K_\C$-domains $X_\C(\Omega)$ and $X_\C(2\Omega)$ are given 
by 
\begin{align*} X_{\C, \Omega}&=
\{ z\in X_\C : \Re  P(z)>0\}\, \\ 
X_{\C, 2\Omega}&=
\{ z\in X_\C : P(z)\in \C \bs]-\infty, -2]\}\, .\end{align*}

We have to understand the inclusion 
$\Xi\subset X_\C(\Omega) \subset  X_\C(2\Omega)$
better. It turns out that $\Xi$ cannot be enlarged. Here is 
the precise result.

\begin{lemma} Let $G=\Sl(2,\R)$. Then for 
$Y\in 2\Omega\bs \oline \Omega$, 
$$G\exp(iY)\cdot x_0\nsubseteq X_{\C, 2\Omega}\, .$$
More precisely, there exists a curve $\gamma(s)$, $s\in [0,1]$,  in 
$G$ such that the assignment  
$$s \mapsto  \sigma(s)=P(\gamma(s)\exp(iY)\cdot x_o)$$ 
is strictly decreasing with values in $[-2,2]$ such that 
$\sigma(0)=P(x_o)=2$ and 
$\sigma (1)=-2$.   
\end{lemma}

\begin{proof} Let $g=\begin{pmatrix} a &b\\  c& d\end{pmatrix}\in G$ and 
$z=\begin{pmatrix} e^{i\phi}& 0\\ 
0& e^{-i\phi} \end{pmatrix}\in \exp (2i\Omega)\setminus 
\exp(i\oline \Omega)$. 
This means $a,b,c,d\in \R$ with $ad-bc=1$ and 
$\frac{\pi}{ 4}<|\phi|<\frac{\pi}{ 2}$ for $\phi\in \R$.  
Thus 
\begin{eqnarray*} p(gz\cdot x_o)&=&p\begin{pmatrix} 
ae^{i\phi} & be^{-i\phi}\\ 
c e^{i\phi} &de^{-i\phi} \end{pmatrix}=
a^2e^{2i\phi} +b^2e^{-2i\phi} +c^2 e^{2i\phi} +d^2e^{-2i\phi}\cr 
&=&\cos (2\phi) (a^2+b^2 +c^2 +d^2) +i \sin 2\phi
(a^2-b^2 +c^2 -d^2)
\end{eqnarray*}

Using that $G=KAN$ and that $p$ is left $K$-invariant, we may actually 
assume that $g\in AN$, i.e. 
$$g=\begin{pmatrix} a& b\\  0 & \frac{1}{a}\end{pmatrix}$$
for some $a>0$ and $b\in \R$. Then 

$$p(gz\cdot x_o)=\cos (2\phi) (a^2+\frac{1}{a^2} +b^2) +i \sin 2\phi
(a^2-\frac{1}{a^2}-b^2 ).$$
We now show that $p(gz\cdot x_o)=-2$ has a solution for fixed  
$\frac{\pi}{ 4}<|\phi| <\frac{\pi}{ 2}$. 
This is because 
$p(gz\cdot x_o)=-2$ forces 
$\Im p(gz\cdot x_o)=0$ and so $b^2=a^2-\frac{1}{ a^2}$. 
Thus 
$$p(gz\cdot x_o)=2 a^2 \cos (2\phi) =-2.$$ 
Thus if we choose $a=\frac{1}{\sqrt{-\cos 2\phi}}$ we obtain a solution. 
The desired curve $\gamma(s)$ is now given by 
$$\gamma(s)=\begin{pmatrix} a(s) &b(s)\\  0& \frac{1}{a(s)}\end{pmatrix}$$
with 
$$a(s)= \frac{1}{\sqrt{-\cos 2\phi}} 
(\sqrt{-\cos 2\phi} + s(1 -\sqrt{-\cos 2\phi})) $$
 and 
$$b(s)=\sqrt{a(s)^2-\frac{1}{a(s)^2}}\, .$$ 
\end{proof}

We are ready for the 

\par\noindent{\bf Proof of Theorem \ref{t=he}.}
We first observe from our previous discussion that 
there exists 
a holomorphic function $\Phi_\lambda$ on $\C \bs (-\infty, 2]=
p(X_{\C,2\Omega})$ 
such that 
\begin{equation} \label{mumu}\phi_\lambda(z)=\Phi_\lambda(P(z)) \qquad 
(z\in X_{\C, 2\Omega}).\end{equation}
Let $Y\in 2\Omega\bs \oline \Omega$. Let 
$\gamma\subset G$ and $\sigma\subset [-2,2]$ be curves as in the previous
lemma.
\par Note that $\gamma(s)\exp(iY)\cdot x_o\subset G$
for all $s\in [0,1)$. Hence (\ref{mumu}) gives 
$$\varphi_\lambda(\gamma(s)\exp(iY)\cdot x_o)=\Phi_\lambda(\sigma(s))
\qquad (s\in [0,1)\, .$$
Now recall that 
$s\mapsto \Phi_\lambda(\sigma(s))$ is positive by (\ref{s=pos}) 
and tends to infinity for $s\nearrow 1$ (cf.\ \cite{KSII}, Th. 2.4).  
Let now $\Xi\subset \Xi'$ be a $G$-domain in $X_\C$ which strictly 
contains $\Xi$. Thus $\partial\Xi\cap \Xi'\neq \emptyset$. 
We recall that $\partial\Xi=\partial_d\Xi\cup \partial_u\Xi$
and distinguish two cases. 
\par\noindent {\it Case 1: $\partial_d\Xi\cap\Xi'\neq \emptyset$.}  
In this case $\Xi'$ contains a point 
$\exp(i2\Omega\bs \oline \Omega)\cdot x_0$ and we arrive at a contradiction. 
\par\noindent {\it Case 2:  $\partial_n\Xi\cap\Xi'\neq \emptyset$.}
This means that $\begin{pmatrix}  1& it \\ 0 & 1\end{pmatrix}\in \Xi'$ for some $t$ with 
absolute value sufficiently close 
to $1$ by (\ref{b=u}). 

With $a_r=\begin{pmatrix} r & 0 \\ 0 & {1\over r}\end{pmatrix}\in A$, $r>0$, and
$-1< t< 1$  that
$$p\left(a_r\begin{pmatrix}  1& it \\ 0 & 1\end{pmatrix}.x_0\right)=r^2 +{1\over r^2} - t^2 r^2\, .$$
In particular, if $|t|>1$, then there would exist a sequence $r_n\to r_0$
such that $p\left(a_{r_t}\begin{pmatrix}  1& it \\
0 & 1\end{pmatrix}\right)\to -2^+$. We argue as before. 
\qed

\subsection{Holomorphic $H$-spherical vectors} 
To begin with I want to explain a few things 
on spherical representations first. 
Throughout this section we let $(\pi,\H)$ be 
an irreducible unitary representation of $G$. For a
subgroup $L<G$ we write $\H^L\subset\H$ for the subspace 
of $L$-fixed elements. As a consequence of the Riemann-Lebesgue
Lemma for representations we obtain: 

\begin{lem} If $L<G$ is closed and non-compact and $\pi$ is non-trivial, then 
$\H^L=\{0\}$.
\end{lem}

So why is this of interest. In case of finite groups, Frobenius 
reciprocity tells us that $\pi$ can be realized in functions
on $G/L$ if and only if $\H^L\neq\{0\}$. 
For non-compact continuous groups we need a more 
sophisticated  version of Frobenius reciprocity:  the Hilbert space 
$\H$ is simply too small for carrying $L$-fixed elements. 
We enlarge $\H$. Recall the space of analytic vectors
$\H^\omega$ of $\pi$. This is a locally 
convex topological vector space of compact type, i.e. a 
Hausdorff direct limit space with compact inclusion 
maps.  We form $\H^{-\omega}$, the strong anti-dual 
of $\H^\omega$, i.e. the space of continuous 
anti-linear functionals $\H^\omega\to \C$ endowed 
with the strong topology. As a topological vector space 
$\H^{-\omega}$ is nuclear Fr\'echet. In particular 
it is reflexive, i.e. its strong anti-dual gives 
us $\H^\omega$ back. 
We note that $\H$ is naturally included 
in $\H^{-\omega}$ via $v\mapsto \langle \cdot, v\rangle$
and obtain the reflexive sandwiching 

$$\H^\omega\hookrightarrow \H\hookrightarrow \H^{-\omega}$$
with all inclusions $G$-equivariant and continuous. 
Sometimes one calls $(\H^\omega, \H, \H^{-\omega})$ 
a {\it Gelfand triple}. 

Now for $G=\Sl(2,\R)$ and $H=\SO(1,1)$ there is the dimension bound 
$$\dim (\H^{-\omega})^H \leq 2\, .$$
To be more precise, for highest or lowest weight representations
the dimension is zero or 1 depending on the parity of the 
smallest $K$-type. For the principal series the dimension is 
$2$. 

\begin{ex} For a principal series representation $\pi_\lambda$ the space 
of $H$-fixed hyperfunction vectors is given by 
$(\H^{-\omega})^H=\operatorname{span}_{\C}\{ \eta_1,\eta_2\}$
with 
$$\eta_1 (x)=\begin{cases} {1\over \sqrt{\pi}}\cdot  {1\over (1-x^2)^{{1\over 2}(1 -i\lambda)}}
& \hbox {for}\ |x|<1, \\
0  & \hbox {for}\ |x|\geq 1;\end{cases}$$
and 
$$\eta_2 (x)=\begin{cases} {1\over \sqrt{\pi}}\cdot  {1\over (x^2-1)^{{1\over 2}(1 -i\lambda)}}
& \hbox {for}\ |x|>1, \\
0  & \hbox {for}\ |x|\leq 1.\end{cases}$$
\end{ex} 

We take a closer look at the basis $\{ \eta_1, \eta_2\} $ in the previous 
example. For what follows it is useful to compactify $\R$ to $\P^1(\R)=G/MAN$ and view 
$\H$ as a function space on $\P^1(\R)$.  Then both $\eta_1$ and $\eta_2$ 
are supported on the two open $H$-orbits in $\P^1(\R)$, namely 
$(-1,1)$ and $\P^1(\R)\bs [-1,1]$. Thus $\eta_1$, $\eta_2$ appear 
to be natural in view of the natural $H$-action on the flag variety. 
However, we claim that it is not the natural basis
for $(\H^{-\omega})^H$. Why? Simply because it is not invariant under 
intertwining operators -- intertwiners here are pseudo-differential operators 
which do not preserve supports. So it is our aim to provide 
a natural basis for the $H$-sphericals. For that our theory 
of holomorphic extension of representations comes 
handy. 
\par Our motivation comes from finite dimensional 
representations. 
\subsubsection{Finite dimensional spherical representations}
Let $(\rho, V)$ be a representation 
of $G$ on a finite dimensional complex vector space 
$V$. Then $\rho$ naturally extends to 
a holomorphic representation of $V$, also denoted by 
$\rho$, and observe: 
$$V^K= V^{K_\C}\qquad \hbox{and}\qquad V^H=V^{H_\C}\, .$$
Here is the punch line: While $H$ and $K$ are not 
conjugate in $G$ (one is non-compact, one is compact), their 
complexifications $H_\C$ and $K_\C$ are conjugate in $G_\C$. 
With 
$$z_H=\begin{pmatrix} e^{i\pi/4} & 0 \\ 0 & e^{-i\pi/4}\end{pmatrix}$$
there is the identity: 
$$z_H H_\C z_H^{-1}= K_\C\, .$$ 
Therefore the map 
\begin{equation} \label{hi} V^K\to V^H, \ \ v\mapsto \rho(z_H) v \end{equation}
is an isomorphism. 
\subsubsection{Construction of the holomorphic $H$-spherical vector}
Our goal here is to find an analogue of  (\ref{hi}) for infinite 
dimensional representations. For what follows we assume in addition that 
$(\pi,\H)$ is $K$-spherical and fix 
a normalized generator $v_K\in \H^K$. 
Now, observe that $z_H\cdot x_0\in \partial_d\Xi=Y=G/H$. 
For $\e>0$ we set 
$$a_\e:=\begin{pmatrix} e^{i(\pi/4-\e)} & 0 \\ 0 & e^{-i(\pi/4 -\e)}\end{pmatrix}$$
and remark: 
$$\lim_{\e\to 0} a_\e= z_H \qquad \hbox{and}\qquad a_\e\in \Xi K_\C\, .$$ 
In particular $\pi (a_\e)v_K$ exists  for all $\e>0$ small. It is no surprise 
that the limit exists in $\H^{-\omega}$ and is $H$-fixed. 
In fact it is a matter of elementary functional analysis 
to establish the following theorem, see  \cite{GKO}, Th. 2.1.3 for a
result in full generality.  

\begin{thm} Let $(\pi,\H)$ be a unitary irreducible representation of 
$G$. Then the map 
$$\H^K \to (\H^{-\omega})^H, \ \ v_K\mapsto v_H:=\lim_{\e\to 0} 
\pi(a_\e)v_K$$
is defined and injective. 
\end{thm}

We call the vector $v_H$ the {\it $H$-spherical holomorphic 
hyperfunction vector} of $\pi$. It is natural in the sense that 
it is preserved by intertwining (observe that intertwiners commute 
with analytic continuation). We will return to this topic later 
when we discuss the most continuous spectrum of $L^2(Y)$. 

\par We wish to make $v_H$ explicit for the principal 
series $\pi_\lambda$. A simple calculation gives 

$$v_H= e^{-i{\pi\over 4} (1-\lambda)}\eta_1 + e^{i{\pi\over 4}(1-\lambda)} \eta_2\, .$$
Upon conjugating the coefficients we get a second, linearly 
independent vector 

$$\oline{v_H}= e^{i{\pi\over 4} (1-\lambda)}\eta_1 + e^{-i{\pi\over 4}(1-\lambda)} \eta_2\, .$$
which we call the anti-holomorphic $H$-spherical vector. 
Likewise one obtains $\oline{v_H}$ by using $\oline{z_H}=z_H^{-1}$ instead of $z_H$. It 
features the same invariance properties as $v_H$. We therefore 
arrive at a basis 
$$\{v_H, \oline{v_H}\}$$
 of $(\H^{-\omega})^H$ which is 
invariant under intertwining, i.e. a canonical diagonalization 
of scattering in the affine symmetric space $Y$.  

\section{Growth of holomorphically extended orbit maps}
Throughout this section $(\pi,\H)$ is a unitary 
irreducible representation of $G$ and $v=v_K\in \H^K$ a 
normalized $K$-finite vector. 
Our objective of this section is to discuss  
the growth of the orbit map 
$$f_v: \Xi \to \H,\ \  zK_\C \mapsto \pi(z)v$$
for $z$ approaching the boundary of $\Xi$.  
We are interested in two quantitities: 
\begin{itemize} \item The norm of $\|\pi(z)v\|$ for 
$z\to \partial\Xi$. 
\item The invariant Sobolev norms $S_k^G (\pi(z)v)$
for $z\to \partial\Xi$.
\end{itemize}
The invariant Sobolev norms were introduced by Bernstein and 
Reznikov in \cite{BR} as a powerful tool to give growth estimates 
for analytically continued automorphic forms. We will comment more 
on that in the subsections  below. 
\par We notice that 
$$\|f_v(g\exp(iY)\cdot x_0)\|=\|\pi(\exp(iY))v\|$$
for all $g\in G$ and $Y\in\Omega$. 
Thus for our growth- interest for $z\mapsto \partial\Xi$ we may 
assume that $z=\exp(iY)\cdot x_0$ for $Y\to \partial\Omega$, or 
with our previous notation 
with $Z=a_\e\cdot x_0$ for $\e\to 0$. 

\subsection{Norm estimates}
Here we determine the behaviour of

$$\|\pi(a_\e)v\|\qquad \hbox{for $\e \to 0$}\, .$$

For $G=\Sl(2,\R)$ this is a simple matter -  for general $G$ 
this is a serious and difficult problem; it was settled in \cite{KO}.  

\begin{prop}\label{p=est} Let $(\pi,\H)$ be a unitary $K$-spherical 
representation of $G$ and $v$ a normalized $K$-fixed vector. Then 
$$\|\pi(a_\e)v\|\asymp \sqrt{|\log \e|}$$
for $\e\to 0$. 
\end{prop}

\begin{proof} It is no big loss of generality 
to assume that $\pi=\pi_\lambda$. Within the non-compact 
realization we determine: 

\begin{align*} \|\pi(a_\e)v\|^2 &={1\over \pi} 
e^{\lambda \pi/2}\int_\R  
\left|{1\over (1+ e^{-i(\pi-4\e)} x^2)^{{1\over 2}(1+i\lambda)}}\right|^2
\ dx \, ,\\ 
&\asymp \int_{-2}^2 \left|{1\over (1+ (-1 +i\e)  x^2)}\right|
\ dx\, , \\ 
&\asymp \int_0^1 {1\over (|u|+ \e)} 
\ du\, , \\ 
& \asymp |\log \e| \, .\end{align*}
\end{proof}

I want to pose the following 

\bigskip\noindent {\bf Problem:} {\it Fix $\sigma\in \hat K$ and let 
$\H(\sigma)$ be the corresponding $K$-type. Determine optimal bounds for 
$$\|\pi(a_\e)v|| \qquad (v\in \H(\sigma))$$
for $\e \to 0$. Possibly generalize to all semi-simple groups.} 

\bigskip 

\subsection{Invariant Sobolev norms}

We first recall some definitions from \cite{BR}.

\begin{df} {\rm (Infimum of seminorms; cf. \cite{BR}, Appendix A)}
Let $V$ be a complex vector space and 
$N_{i\in I}$ a family of semi-norms. Then the prescription 
$$\inf_{i\in I} N_i (v):=\inf_{v =\sum_{i\in I} v_i } \sum_{i\in I} 
N_i(v_i)$$
defines a semi-norm. It is the largest seminorm with respect 
to the property of being dominated by all $N_i$. 
\end{df}

\begin{rem} To get an idea of the nature of the definition of 
the infimum seminorm $\inf N_i$ it is good to think in the following 
analogy: Think of $V$ as a function space, say on $\R$ and think of 
$N_i$ as a seminorm with support on a certain interval, say  $J_i$. 
such that $\cup J_i =\R$. Further $v=\sum_{i\in I} v_i$ should be 
considered as breaking the function $v$ into functions $v_i$ 
with smaller support in $J_i$. 
\end{rem}
  
We want to bring in a symmetry group $G$ which acts linearly 
on the vector space $V$. We start with one seminorm $N:V\to 
\R_\geq 0$ and produce others: for $g\in G$ we let 
$$N_g(v):=N(g(v))\, .$$
In this way we obtain a seminorm 
$$N^G:=\inf_{g\in G} N_g(v)$$ 
which is uniquely characterized as being the 
largest $G$-invariant seminorm on $V$ which is 
dominated by $N$. 

\par We come to specific choices for $V$ and $N$. 
For $V$ we use the Fr\'echet-space of smooth vectors $\H^\infty$
for the representation $\pi$; the seminorm $N$ will be 
Sobolev norm. 
We briefly recall their construction. 
Recall that the derived representation $d\pi$ of $\gf$ is defined 
as 
$$d\pi: \gf\to \operatorname{End}(\H^\infty), \ \ d\pi(Z)(v):=
{d\over dt}\Big|_{t=0} \pi (\exp(tZ))v \, .$$ 
We fix a basis 
$Z_1, Z_2, Z_3$ of $\gf$  and an integer $k\in \N_0$. Then the $k$-th 
Sobolev norm $S_k$ of $\pi$ is defined as 

$$S_k(v):=\sum_{k_1+k_2 +k_3\leq k} \|d\pi(Z_1)^{k_1}  
d\pi(Z_2)^{k_2}d\pi(Z_3)^{k_3}v\|\qquad (v\in \H^\infty)\, .$$    
Let us emphasize that $S_k$ depends on the 
chosen basis $Z_1, Z_2, Z_3$, but a different basis 
yields an equivalent norm. Our interest is now with 
$S_k^G$ the $G$-invariant Sobolev norm. 
Notice that $S_0^G=\|\cdot\|$ is the Hilbert norm, as 
we assume that $\pi$ is unitary. 
In view of our preceding remark it is natural to view 
$S_k^G$ as some Besov-type norm for the representation. 

\par We wish to understand the nature of $S_k^G$. 
For that it is useful to introduce the following 
notation:  For a closed subgroup $L<G$ we write 
$S_{k,L}$ for the $k$-th Sobolev norm for the 
restricted representation $\pi|_L$. 
We make a first simple observation: 

\begin{lem}\label{l=pr} Let $(\pi,\H)$ be a unitary representation of $G$ and
$v\in\H^\infty$. Then for all $k\geq 0$:
\begin{enumerate}
\item $S_{k,N}^{A^+}(v)= \|v\|$. 
\item $S_{k,AN}^G (v)=S_k^G(v)$. 
\end{enumerate}
\end{lem}

\begin{proof} Easy; see \cite{KSI}, Lemma 6.5 for the 
general statement. 
\end{proof}

The following Theorem is fundamental (\cite{KSI}, Prop. 6.6). 

\begin{thm} Let $(\pi,\H)$ be an irreducible unitary 
representation of $G$. Let $k\in \Z_{\geq 0}$. Then there 
exists a constant $C=C(k,\pi)$ such that 
$$S_k^G(v)\leq C \cdot S_{k,A}^G(v) \qquad (v\in\H^\infty)\, .$$
\end{thm}

\begin{proof} We will only treat the case of $\pi=\pi_\lambda$. 
We remark that 
$$\H^\infty=\{ f\in C^\infty(\R)\: |x|^{i\lambda -1} f({1\over 
x})\in C^\infty(\R)\} $$
and introduce some standard notation

We use a usual basis for the Lie algebra of $\g$  
$${\bf h} =\begin{pmatrix} 1 & 0 \\ 0 & -1\end{pmatrix} , \qquad {\bf e}=\begin{pmatrix} 0 & 1\\ 0 & 0
\end{pmatrix}, 
\qquad {\bf f}=\begin{pmatrix} 0 & 0 \\ 1& 0\end{pmatrix}\, .$$
Then $\af=\R {\bf h}$, $\nf=\R {\bf e}$ and $\oline \nf=\R {\bf f}$. 
With ${\bf u} ={\bf e}-{\bf f} $ we have $\kf=\R {\bf u}$. 
Differentiating the action (\ref{(6.1)}) one obtains the formulas 

\begin{align} \label{(6.2)}
d\pi_\lambda({\bf h})&= (i\lambda -1 ) -2 x{d\over dx}\,, \\ 
\label{(6.3)} d\pi_\lambda({\bf e})&= - {d\over dx}\,,  \\  
\label{(6.4)} d\pi_\lambda({\bf f})&= (1-i\lambda)x + x^2 {d\over dx}\,,\\
 \label{(6.5)} d\pi_\lambda({\bf u})&= (i\lambda-1) - (1+x^2) {d\over dx}\,, \\ 
 \label{(6.6)} d\pi_\lambda({\bf e}+{\bf f})&= (1-i\lambda)x - (1-x^2) {d\over dx}\, .
\end{align}

We also define the {\it radial  operators} 
by 
$$(R_jf)(x)=(x^j {d^j\over dx^j}f)(x)$$
and define the {\it radial Sobolev norms} by 

$$S_{k, {\rm rad}}(f)=\sum_{j=0}^k \|R_j f\|.$$
{}From the action of $d\pi_\lambda({\bf h})$ and $R^j$ it is clear that 
there exists  a constant $C>0$, depending on $k$ and $\lambda$, such that for 
all $f\in {\mathcal S}(\R)$
\begin{equation} \label{(6.7)} {1\over C} S_{k,{\rm rad}} (f)\leq 
S_{k,A}(f)\leq C S_{k, {\rm rad}}(f)\, .\end{equation} 

We wish to point out that in 
(\ref{(6.2)}) and (\ref{(6.4)}) the coefficient of the derivative 
term has a zero, consequently $S_k(v)$ can not be majorized by $S_{k,A\oline 
N}(v)$ or by $S_{k,A}(v)$ in general. However, we shall show in the next 
Proposition that there is such a relationship for the $G-$invariant Sobolev 
norms.  

\par The $A$ action on $K/M\cong S^1$ has two fixed points, corresponding to 
the two Bruhat cells. In the non-compact realization $N$ they become the origin 
and the point at infinity. We shall estimate $ S_k^G (f)$ by using first a 
cutoff function at infinity, $\oline \nf$, and an elementary estimate there. Near 
the origin a dilated cutoff localizes sufficiently high derivatives of $f$ to 
get an estimate. Away from the fixed points, motivated by an argument in \cite{BR} 
 and classical Littlewood-Paley theory, we use a family of suitably dilated cutoff functions 
which compress the $\nf$ derivatives in the definition of $G$-invariant norm to {\it radial} 
derivatives thereby obtaining the desired 
estimate. 

For $j\in \Z$ we denote by $I_j$ the set $\{ x\in\R \: 2^{-j-1}\leq |x|\leq 
2^{-j+1}\}$. 
For a function $\psi$ on $\R$ we write $\psi_j(x)=\psi(2^j x)$. Notice that if 
$\psi$ is supported in $I_0$ then $\psi_j$ is supported in $I_j$, and 
$${\rm supp}(\psi_j)\ \cap \ {\rm supp} (\psi_{j+1}) \subseteq \lbrack {1\over 
2^{j+1}}, {1\over 2^j}\rbrack.$$  
We take a smooth, non-negative function $\phi$ supported in $I_0$ and such that 
for every $m\in \N_0$ we have 
$$\sum_{j=0}^m \phi_j(x)=\begin{cases} 0 & \hbox{if}\   |x|\leq  2^{-m-1}\, ,\\
1 & \hbox{if}\  2^{-m}\leq |x|\leq 1\,, \\ 
0 & \hbox{if} \  2\leq |x|\end{cases}$$

\par Choose a nonnegative function $\tau\in C^\infty(\R)$ with support in $\{ 
x\in \R\: 1\leq |x|\}$ 
such that $(\tau+\phi)(x)=1$ for $|x|\geq 1$. 
Finally for each $m\in \N$ define the function $\tau_m\in C_c^\infty(\R)$ by 
$\tau_m ={\bf 1} -\tau-\sum_{j=0}^m \phi_j$. Notice that $\operatorname{supp} 
\tau_m\subset \{x\in\R\mid   |x|\leq 2^{-m}\}$
and $\tau_m(x)=1$ for $|x|\leq 2^{-m-1}$. From the properties of the $\phi_j$ 
and $\tau$ it is easy to see that for any $l\geq 1$, $\tau^{(l)}_m (x) = 
-2^{lm}\phi^{(l)}(2^mx)$.
\par Let $f\in \H^\infty$. Since
\begin{align*} {\bf 1} & =  \tau + {\bf 1} - \tau\cr & = \tau + \tau_m + \sum_{j=0}^m 
\phi_j\\
& = \tau + \phi + \tau_m + \sum_{j=1}^m \phi_j\, ,\end{align*}

then $$ f = (\tau + \phi)f + \tau_mf + \sum_{j=1}^m {\phi_jf}. $$
For any choices of $g, g_1,\ldots, g_m\in G$, using the definition of $S_k^G$, 
we get
\begin{equation} \label{(6.8)} S_k^G(f)\leq S_k((\tau +\phi) f) + S_k(\pi_\lambda(g)(\tau_m f)) + \sum_{j=1}^m 
S_k(\pi_\lambda(g_j) (\phi_j f))\, . \end{equation}
\par First we consider the term $S_k((\tau +\phi)f)$. From an examination of 
formulas (\ref{(6.2)}) - (\ref{(6.4)}) one sees that $S_k((\tau +\phi) f)\leq C S_{k,\oline 
N}((\tau +\phi)f)$
for all $f\in \H^\infty$. 
(Throughout this proof $C$ will denote a constant depending only on $k$, $\tau$, $\phi$ and $\lambda$.)
Hence we have 
$$S_{k}((\tau +\phi)f)\leq CS_{k,\oline N}((\tau +\phi) f)\leq C S_{k,\oline 
N}(f)$$
for all $f\in \H^\infty$.  Majorizing this term in (6.8) we get 
\begin{equation} \label{(6.9)} S_k^G(f)\leq C S_{k,\oline N}(f) + S_k((\pi_\lambda(g)\tau_m f))+ 
\sum_{j=1}^m S_k(\pi_\lambda(g_j) (\phi_j f))\end{equation}
for all $f\in \H^\infty$. 

\par Next we specify a good choice of the elements 
$g, g_1,\ldots, g_m\in G$. For every $t>0$ denote by $b_t$ the element 
$$b_t =\begin{pmatrix} {1\over \sqrt t} & 0 \\ 0 & \sqrt t\end{pmatrix} \in A.$$
From (\ref{(6.1)}) it follows that 
$$(\pi_\lambda(b_t)f)(x)= t^{{1\over 2}(1-\lambda)} f(tx)$$
for all $t>0$ and $x\in \R$. 
Take $g_j=b_{2^{-j}}$ for all $1\leq j\leq m$ and $g = b_{2^{-(m+1)}}$. 
Notice that for every $m$ all the $\pi_\lambda(g_j) (\phi_j f)$ are  
supported in $[-2,2]$, as is $\pi_\lambda(g) (\tau_m f)$. For any smooth 
function $h$ supported in $[-2,2]$ we can conclude from the formulas (\ref{(6.2)}) - 
(\ref{(6.5)}) that
$S_k(h)\leq CS_{k,N}(h)$. Using this in (6.9) we get
$$S_k^G(f)\leq C S_{k,\oline N}(f) +C S_{k,N}(\pi_\lambda(g) (\tau_mf)) + 
C\sum_{j=1}^m S_{k,N}(\pi_\lambda(g_j) (\phi_j f))\leqno(6.10)$$
for all $f\in \H^\infty$. 
\par We estimate  $S_{k,N}(\pi_\lambda(g) (\tau_mf))$.  For this we use Leibniz 
on $\tau_m f$ and $L^\infty$ estimates on $\tau_m^{(j)} = 
-2^{jm}\phi^{(j)}(2^mx)$. From (\ref{(6.3)}) one sees that $S_{k,N}(h)=\sum_{l=0}^k 
\|h^{(l)}\|$. Then
\begin{align*}S_{k,N}&(\pi_\lambda (g) (\tau_m f))  = \sum_{l=0}^k  \| 
{d^l\over dx^l} 2^{-{(m+1)\over 2}(1-\lambda)} (\tau_m f)(2^{-(m+1)}\cdot)\|\\
& =\sum_{l=0}^k \vert 2^{-{(m+1)\over 2}(1-\lambda)}\vert\Big[
\int\Big|\sum_{n=0}^l 2^{-(m+1)l}{l \choose l-n} \cdot \\ 
& \qquad \cdot \tau_m^{(l-n)}(2^{-(m+1)}x) f^{(n)}
(2^{-(m+1)}x)\Big|^2 \ dx\Big]^{1\over 2}\\
& \leq\sum_{l=0}^k \vert 2^{-{(m+1)\over 2}(1-\lambda)}\vert \sum_{n=0}^l
\Big[\int_{|x|\leq 2}\Big| 2^{-(m+1)l}{l \choose l-n}\cdot \\ 
& \qquad \cdot  \tau_m^{(l-n)}(2^{-(m+1)}x) 
f^{(n)}(2^{-(m+1)}x)\Big|^2 \ dx \Big]^{1\over 2}\\
& =\sum_{l=0}^k \Big|2^{{(m+1)\over 2}\lambda}\Big| \sum_{n=0}^l
\Big[\int_{|y|\leq {1\over 2^m}} \Big|2^{-(m+1)l} {l \choose l-n} \tau_m^{(l-n)}(y) f^n(y) \Big|^2\  dy 
\Big]^{1\over 2}\\
& \leq\sum_{l=0}^k | 2^{{(m+1)\over 2}\lambda}| \sum_{n=0}^l{l \choose l-n}{ \|2^{(l-n)m}\phi^{(l-n)}\|_\infty\over 2^{(m+1)l}}
\Big[\int_{|y|\leq {1\over 2^m}} |f^{(n)} (y) |^2\  dy \Big]^{1\over 2}\\
& =\sum_{n=0}^k | 2^{{(m+1)\over 2}\lambda}| {1\over 2^{mn}}
\sum_{l=n}^k{l \choose l-n}{\|\phi^{(l-n)}\|_\infty\over 2^l} \Big[\int_{|y|\leq {1\over 2^m}} 
|f^{(n)}(y)|^2\ dy\Big]^{1\over 2}\\
& =\sum_{n=0}^k | 2^{{(m+1)\over 2}\lambda}| {1\over 2^{(m+1)n}}
\sum_{j=0}^{k-n}{j+n \choose n} {\|\phi^{j}\|_\infty\over 2^j} \Big[\int_{|y|\leq {1\over 2^m}} 
|f^{(n)}(y)|^2\ dy\Big]^{1\over 2}\\
& \leq\big(\sum_{j=0}^k{\|\phi^{(j)}\|_\infty\over {j!2^j}}\big)\sum_{n=0}^k 
{k!\over 
n!2^{(m+1)n}}\Big[\int_{|y|\leq {1\over 2^m}}| f^{(n)}(y)|^2 \ dy\Big]^{1\over 
2}.\end{align*}
Now $k$ is fixed and each of the at most $k$ derivatives $f^{(n)}$ is in $L^2$, 
hence the integrals can be made uniformly small. So for each $f$ we can choose 
an $m$ so that the last line above is at most $\|f\|$. Then we have
 
$$S_k^G(f)\leq C S_{k,\oline N}(f) + C\|f\| + 
C \sum_{j=1}^m S_{k,N}(\pi_\lambda(g_j) (\phi_j f))$$
for any $f\in \H^\infty$.   Thus 
 we obtain   that 
\begin{equation} \label{(6.12)} S_k^G(f)\leq C S_{k,\oline N }(f) + C\|f\| +
C\sum_{l=0}^k \sum_{j=1}^m \| {d^l\over dx^l} (2^{-{j\over 2}(1-i\lambda)}\phi 
f(2^{-j}\cdot))\|\, .\end{equation} 
As in the long computation above,  
using Leibniz on $\phi f$, $L^\infty$ estimates on $\phi^{(j)}$, and majorizing the binomial coefficients we get
\begin{align*}\sum_{l=0}^k \sum_{j=1}^m \| {d^l\over dx^l} (2^{-{j\over 2}}\phi 
f(2^{-j}\cdot))\|
& \leq C 
\sum_{l=0}^k \sum_{j=1}^m \Big(\int_{I_0}  2^{-j-2l}|f^{(l)} (2^{-j}x)|^2 \ 
dx\Big)^{1\over 2}\\ 
& =  C \sum_{l=0}^k \sum_{j=1}^m \Big(\int_{I_j} 2^{-2l} |f^{(l)} (x)|^2 \ 
dx\Big)^{1\over 2}\\ 
&\leq 4  C \sum_{l=0}^k \sum_{j=1}^m \Big(\int_{I_j}  | x^lf^{(l)} (x)|^2 \ 
dx\Big)^{1\over 2}\\ 
&\leq 4  C S_{k,{\rm rad}} (f) \leq 4 C S_{k,A} (f), 
\end{align*}
where the last inequality follows from (\ref{(6.7)}) and again $ C$ depends only on $\tau$, $\phi$, $k$ and $\lambda$. 
Thus we get from (\ref{(6.12)}) and (\ref{(6.13)}) that 
$$S_k^G (f) \leq C S_{k,\oline N}(f) +  C\|f\|  + C S_{k,A}(f)\leq  C\|f\| + C 
S_{k,A\oline N}(f)$$
for all $f\in \H^\infty$. Thus 
$$S_k^G \leq C S_{k,A\oline N}^G$$ and, using Lemma \ref{l=pr}(ii),  $S_k^G \leq C 
S_{k,A}^G$ as was to be shown.
\end{proof}

With regard to the above theorem I want to pose 
the following 

\bigskip\noindent  
{\bf Problem:} {\it Formulate and possibly prove the above 
result for all semisimple groups.}

\bigskip We come to the main result of this section, see \cite{KSI}, Th. 6.7: 
the estimate for $S_k^G(\pi(a_\e)v)$.  
We will only explain the idea and refer to \cite{KSI} for a discussion 
in full detail. 
We fix on the case $\pi=\pi_\lambda$ and observe 
that, up to constant: 
$$[\pi(a_\e)v](x)= {1\over (1+e^{i\pi(1-\e)}x^2)^{{1\over 2}(1-i\lambda)}}\qquad (x\in\R)$$ 
Hence $\pi(a_\e)v(x)$ develops singularities at 
$x=\pm 1$ which are logarithmic in the $L^2$-sense, see Proposition \ref{p=est} from above.  
Taking the $k$-th Sobolev norm increases the singularity 
accordingly; one verifies for $k\geq 1$ that 

$$S_k(\pi(a_\e)v )\asymp \e^{-k}\, .$$
It is so remarkable that the situation is much different 
for $S_k^G (\pi(a_\e)v$. Why? Observe that 

\begin{equation} \label{e=ah} S_{k,H}(\pi(a_\e)v)\asymp \|\pi(a_\e)v\|\end{equation}
as the fixed points of $H$ are precisely $x=\pm 1$, the 
loci where the function $\pi(a_\e)v$ develops 
singularities (cf. with (\ref{(6.6)})). 
Now with 
$$k_0={1\over\sqrt{2}} \begin{pmatrix} 1 & 1 \\ -1 & 1\end{pmatrix}\in K$$
there is an element which rotates $\af$ to $\h$. Hence 
$$S_{k,A} (\pi(k_0)\pi(a_\e)v)= S_{k,H}(\pi(a_\e)v)\, $$
and combined with (\ref{e=ah}) we arrive 
at the hardest result in this article. 

\begin{thm}\label{t=so} Let $(\pi,\H)$ be a unitary irreducible 
representation of $G$ and $v\in\H$ a $K$-fixed vector.
Let $k\in\Z_{\geq 0}$.  
Then there exists a constant $C= C(\pi, k)$ such that 
$$S_k^G (\pi(a_\e)v)\leq C \|\pi(a_\e)v\|$$
for all $\e>0$ small. 
\end{thm}

I expect the theorem from above to be true for all 
$K$-finite vectors $v$ with the reservation that 
$C=C(\pi,K)$ depends on the occuring $K$-types in the 
support of $v$ in addition. In \cite{KSI} we conjecture 
(Conjecture C) that the the estimate holds even for arbitrary 
semisimple Lie groups. This is very difficult. For real 
rank one we could establish this for the $K$-fixed vector 
in \cite{KSI}.

\section{Harmonic analysis on the crown}
\subsection{Holomorphic extension of eigenfunctions}
Let 
$$\Delta= -y^2(\partial_x^2 + \partial_y^2)$$
be the Laplace-Beltrami operator on $X$. For 
$\mu\in\C$ we consider the eigenvalue problem 

$$\Delta \phi = \mu(1-\mu)\phi\, .$$
We observe that solutions $\phi$ are necessarily analytic 
functions as $\Delta$ is an elliptic operator. 
Analytic functions admit holomorphic 
extensions to some complex neighborhood of $X$ in $X_\C$. 
Further, as $G$ commutes with $\Delta$, the resulting domain 
$D_\phi\subset X_\C$ attached to $\phi$ is $G$-invariant. 
By now it should be no surprise that $D_\phi=\Xi$ for generic 
choices of $\phi$. In fact it is just a disguise of the non-unitary 
version of Theorem \ref{t=m}, see  \cite{KSII}, Th. 1.1 and 
Prop. 1.3. 

\begin{thm} All $\Delta$-eigenfunctions on $X$ extend to holomorphic 
functions on $\Xi$. 
\end{thm}
\begin{proof} At this point it would better to switch from $X$ to its 
bounded realization: the unit disc. It has the advantage of circular 
symmetry on a compact boundary and results in a good grip 
concerning convergence problems of boundary value issues on $X$. 
However, I do not want to do that and thus certain convergence 
issues will remain untreated below. 

\par   To begin with we recall the Poisson-kernel $P$ on $X$: 
$$P(z)={1\over \pi} {\Im z\over z\cdot \oline z}\qquad (z\in X)\, .$$
Now if $\Delta \phi= \mu(1-\mu)\phi$ with $\mu\neq 0$, then there is a
generalized function $\phi_\R$ on $\R$ as boundary value 
of $\phi$ from which we can reconstruct $\phi$ via 
Poisson integration: 
$$\phi(z)=\int_\R \phi_\R (x) P^\mu (z-x)\ dx \, .$$
Now observe that $P$ admits a holomorphic extension $P^\sim$  
to $\Xi=X\times \oline X$ obtained by polarization: 

$$P^\sim(z,w)= {1\over 2\pi i} {z-w\over z\cdot w}\qquad ((z,w)\in \Xi)\, .$$
Thus $\phi$ admits a holomorphic extension $\phi^\sim$ to 
$\Xi$ by setting

$$\phi^\sim (z,w)=\int_\R \phi_\R (x) (P^\sim)^\mu (z-x,w-x)\ dx \, .$$
\end{proof}

\subsection{Paley-Wiener revisited}
Let us begin with a short disgression into history: the 
theorems  of Paley and Wiener \cite{PW} on the restriction 
of the Fourier transform to various meaningful function spaces. 
\par When dealing with Fourier analysis on $\R^n$ one often 
identifies $\R^n$ with its dual space. However, it is better not 
to do it in order to avoid confusion between the geometric
and spectral features.   

\par Let $V$ be a finite 
dimensional real vector space $V$. Its dual space shall be 
denoted $V^*$. We fix an Euclidean structure on $V$ (and hence on $V^*$) 
and normalize the 
resulting Lebesgue measures $dv$, $d\alpha$ such that the 
Fourier transform 
$$\F: L^1(V)\to C_0(V^*), \ \ f\mapsto \F(f):=\hat f; 
\ \hat f(\alpha):=\int_V e^{i\alpha(v)} f(v) \ dv \, $$
extends to an isometry $L^2(V)\to L^2(V^*)$. 
\par Actually we wish to view $V^*$ as $\hat V$, the unitary 
dual of the abelian group $(V,+)$. The isomorphism 
is given by 

$$V^*\to \hat V, \ \ \alpha\mapsto\chi_\alpha; 
\ \chi_\alpha(v)=e^{i\alpha(v)}\, .$$

For a general, say reductive, group $G$, we know from 
the work of Segal that there is a Fourier transform 
from $L^1(G)$ to a Hilbert-valued fiber bundle ${\mathcal V}\to 
\hat G_{\rm temp}$
over the tempered unitary dual $\hat G_{\rm temp}$ of $G$ which extends 
to an isometry $\F: L^2(G)\to \Gamma^2({\mathcal V})$. Here 
$\Gamma^2$ stands for the $L^2$-sections of the bundle with respect 
to the Plancherel measure which was determined explicitly by Harish-Chandra, 
 \cite{HC1}.  
  
\par Back to our original setup of $V$ and $V^*$. In the context
of Fourier transform one might ask about the image of 
certain function spaces, for instance test functions, Schwartz 
functions, their duals,  or of functions on $V$ which extend 
holomorphically to some tube domain in $V_\C =V+iV$. 
Paley-Wiener theory is concerned with the first and last 
mentioned examples in the uplisting. 
For a more serious discussion we need more
precision. 

\bigskip {\it The image of test functions.}
We want to characterize $\F(C_c^\infty(V))$. 
For that we define for every $R>0$ the 
subspace $C_R^\infty(V)$ of those test functions which are 
supported in the Euclidean ball of radius $R$. Likewise 
we define $\PW_R(V_\C^*)$ to be the space of those holomorphic 
functions $f$ on $V_\C^*$, the complexification of $V^*$, 
which satisfy the growth condition   
$$|f(\alpha +i\beta)| \ll e^{R \|\beta\|} (1+ \|\alpha\| +\|\beta\|)^{-N}
\qquad (\alpha, \beta\in V^*) $$
for all $N>0$. Then the smooth version of Theorem X   
of Paley and Wiener (cf. \cite{PW}) asserts that 
\begin{equation}\label{PW1} 
\F(C_R^\infty(V))=PW_R(V_\C^*)\qquad(\hbox{PW-I})\, .\end{equation}

\bigskip {\it The image of strip functions.} For $R>0$ we let 
$B_R$ be the ball of radius $R$ centered at the origin and define 
a tube domain in $V_\C$ by 
$$S_R = V + i B_R\, .$$
Further we define 
$$\S_R (V):=\{ f\in \O(S_R)\mid \sup_{w\in B_R} 
\int_V |f(v+iw)|^2 \ dv <\infty\}\, $$  
and simply call them strip functions. 
Then Theorem IV of Paley and Wiener \cite{PW} specializes to  
\begin{equation} \label{PW2} \F(\S_R(V))={\mathcal E}_R(V^*)
\qquad(\hbox{PW-II})\end{equation}
with 
$${\mathcal E}_R(V^*)=\{ f\in L^2(V^*)\mid  \int_{V^*} |f(\alpha)|^2 
e^{2R\|\alpha\|} \ d\alpha< \infty\}$$ 
the space of exponentially  
decaying functions $L^2$-functions on $V^*$ with decay exponent $R$. 

\bigskip We move from $V$ to $G$. As we remarked 
earlier, we have to be careful because of the symmetry break between $G$ and 
$\hat G$. So there are in fact four different types 
of Paley-Wiener theorems which are of interest: (PW-I) and (PW-II) 
and as well their inverse versions for $\F^{-1}$.   
\par Arthur did a case of (PW-I) in \cite{A} when he characterized 
the image of the $K\times K$-finite test functions 
$C_c^\infty(G)_{K\times K}$ under $\F$. 
We emphasize the subspace 
$$C_c^\infty(G)_{K\times {\bf 1}}^{{\bf 1}\times K}
\subset C_c^\infty(G)_{K\times K}$$
of functions which are fixed under right $K$-displacements. These 
functions naturally realize as $K$-finite functions on $X$. 
\footnote{As for analysis on $X$ one should 
think of it as $K$-invariant analysis on $G$.}
Then Arthur's Paley Wiener result gives us 
the image of $C_c^\infty(X)_K$ as certain entire sections 
over the complexification of the spherical unitary dual, i.e. 
$\af_\C^*/\W$. It became the bad habit to restrict even further 
to $K$-fixed functions on $X$ -- this makes the sections
scalar valued and matters reduced to some "Euclidean" Harmonic 
analysis with respect to a specific weighted measure space.   
In this simplified context a Paley-Wiener theorem for the  
inverse of (PW-I) was established for some class of 
examples \cite{P}.
A fully geometric version of the inverse of (PW-I) was recently 
obtained by Thangavelu in \cite{T}, when he showed 
that sections with compact support in a ball correspond 
to holomorphic functions on the crown with a certain 
growth condition related to the size of the support. 
We will not further delve into that 
but focus on (PW-II) instead. 

\par So far the discussion was general, but now I wish to return 
-- for the sake of the exposition -- 
to $G=\Sl(2,\R)$ and the upper halfplane $X$ where 
very concrete formulas hold. For 
$0<R\leq \pi/4$ we define a $G$-domain in $\Xi$ by 

$$\Xi_R=G\exp(i\pi/4(-R,R){\bf h})\cdot x_0\,.$$
For $R=\pi/4$ we obtain the crown and in general 
$(\Xi_R)_R$ is a filtration of $\Xi$ of $G$-invariant 
Stein domains (see \cite{GKII} for the general fact). 
We think of $\Xi_R$ as a strip domain around $X$ 
and define the  analogue of the space of strip 
functions by 

$$\S_R(X):=\{ f\in \O(\Xi_R)\mid  \sup_{r\in (0,R)} \int_G 
|f(g\exp(ir{\bf h})\cdot x_0)|^2 \ dg <\infty\}\, $$
By a theorem of Harish-Chandra , $\F$ identifies 
$L^2(X)$ with 
$$L^2\left (K/M\times i\af^*/\W, d(kM)\otimes 
{\lambda \tanh( \pi \lambda) d\lambda}\right)$$
where we have identified $i\af^*$ linearily with $\R$ subject to 
the normalization that the functional 
$c i {\bf h}\to c$ corresponds to 
$1\in\R$. 
As $\W=\Z_2$ acts as the flip on $\R$ we may safely identify 
$i\af^*/\W$ with $[0,\infty)$. Obviously $K/M$ identifies 
with the unit circle.
The Fourier transform on $G$, restricted to $K$-invariants 
is then given by 

$$\F f(kM, \lambda)=\int_X f(z)\phi_\lambda(k^{-1}z) \ dz$$

The Parseval identity for $G$ reduced to $X$ 
then states that:  
$$\int_X |f(z)|^2 \ dz= \int_{K/M} \int_0^\infty  
|(\F f)(kM,\lambda)|^2 \ d(kM) 
\ \lambda \tanh(\pi\lambda/2) d\lambda\, .$$
If we want to extend this identity by moving the $G$-orbit 
$X$ into $\Xi_R$, i.e. a contour shift, then we 
need the Plancherel theorem for $G$ (and not only of $X$). For a function 
$f\in S_R(X)$, we then get for all $r<R$: 

$$\begin{aligned}
&\int_G |f(g\exp(ir{\bf h})\cdot x_0)|^2 \ dg \\
& \qquad =
\int_{K/M} \int_0^\infty |(\F f)(kM,\lambda)|^2 
\phi_\lambda(\exp(i2r {\bf h}) \ \lambda\tanh(\pi\lambda/2) d\lambda\, .
\end{aligned}\label{Gutz}$$
In \cite{F1} Faraut named  this equality {\it Gutzmer identity} in the 
honour of Gutzmer who, in the 19th century,  investigated 
growth of Fourier coefficients with respect to analytic continuation
, \cite{F1}. 
We emphasize that $\phi_\lambda(\exp(i2r {\bf h})$ is a
positive quantity as we know from the doubling identity 
(\ref{s=pos}).   
\par Let us define the analogue of $\E_R(V^*)$ to be 

\begin{align*}\E_R(\hat G)&=\Big\{ f\in 
L^2\left (K/M\times i\af^*/\W, d(kM)\otimes \lambda\tanh(\pi\lambda/2) d\lambda\right)\mid \\
& \sup_{0\leq r<R}\int_{K/M} \int_0^\infty |(\F f)(kM,\lambda)|^2 
\phi_\lambda(\exp(i2r {\bf h}) \ \lambda \tanh(\pi\lambda/2) d\lambda<\infty\Big\}\, 
\end{align*}
and state the analogue of Theorem IV of Paley and Wiener. 

\begin{thm} For all $0\leq R\leq \pi/4$ 
$$f\in S_R(X)\iff \F(f) \in \E_R(\hat G)\, .$$ 
\end{thm}

To end this section I want to pose the following

\bigskip\noindent  
{\bf Problem:} {\it Formulate and possibly prove geometric 
Paley-Wiener theorems, i.e. (PW2) and inverse of (PW1), 
for $G$.}

\subsection{Hard estimates on extended Maa\ss{} cusp forms}
Let $\Gamma<G$ be a lattice. Then, an analytic function 
$\phi: X\to \C$ is called a Maa\ss{} automorphic form if  
\begin{itemize}
\item $\phi$ is $\Gamma$-invariant,  
\item $\phi$ is a $\Delta$-eigenfunction,  
\item $\phi$ is of moderate growth at the cusps 
of $\Gamma\bs X$. 
\end{itemize} 
We note that the third bulleted item is automatic if 
$\Gamma$ is co-compact, i.e. $\Gamma\bs X$ is compact.

\par A Maa\ss{} form $\phi$ is called a cusp form if 
it vanishes at all cusps of $\Gamma\bs X$, i.e.

$$\int_{N'\cap \Gamma\bs N'} \phi(n'x) dn' = 0 \qquad (x\in X)$$
for all unipotent groups $N'<G$ with $\Gamma\cap N'\neq 
\emptyset$. 

\par From now on we assume that $\phi$ is a cusp form. 
Frobenius reciprocity (see \cite{GGPS} and \cite{BR2} for a
quantitative version) tells us that 

$$\phi(gK)=(\pi(g)v_K, \eta) \qquad (g\in G)$$
for $(\pi,\H)$ a unitary irreducible representation of $G$, 
$v_K\in \H^K$ a normalized $K$-fixed vector and $\eta\in 
(\H^{-\infty})^\Gamma$ a $\Gamma$-invariant distribution 
vector ( in \cite{BR2}it is, perhaps  more appropriately, called 
{\it automorphic functional}).  
It is useful to allow arbitrary smooth vectors $v\in \H^\infty$
and build $\Gamma$-invariant smooth functions $\phi_v$ on $G$ by 

$$\phi_v(g)= (\pi(g)v_K, \eta) \qquad (g\in G)\, .$$
Langland's modification of the Sobolev Lemma for cusp
forms then reads as: 

\begin{equation} \label{so1}
\|\phi_v\|_\infty=\sup_{g\in G} |\phi_v(g)|\leq C \cdot S_2(v)
\qquad (v\in \H^\infty) \end{equation}
for $C>0$ a constant only depending on the 
geometry of $\Gamma\bs G$ (see \cite{BR}, Appendix B for an 
exposition). 
As $\|\cdot\|_\infty$ is $G$-invariant, we deduce from 
\ref{so1}
that 

\begin{equation} \label{so2}
\|\phi_v\|_\infty=\sup_{g\in G} |\phi_v(g)|\leq C \cdot S_2^G(v)
\qquad (v\in \H^\infty) \end{equation}
(cf. \cite{BR}, Section 3). 
One deep observations in \cite{BR} was that 
$S_2^G(v)$ can be considerably smaller as $S_2(v)$, for instance 
if $v=\pi(a_\e)v_K$. 
We combine \label{so2} with Theorem \ref{t=so} and Proposition \ref{p=est} 
(cf. \cite{KSI}, Th. 6.7)

\begin{thm}\label{t=br} Let $\phi$ be a Maa\ss{} cusp form. Then there 
exist constants $C, C'>0$ such that 
$$\sup_{g\in G} |\phi(ga_\e\cdot x_0)| \leq C 
\|\phi(\cdot a_\e)\|_{L^2(\Gamma\bs G)}\leq  C' 
\cdot \sqrt{|\log \e|}$$ 
\end{thm}

\begin{rem} In \cite{BR} a slightly weaker 
bound was established, namely: 
$$\sup_{g\in G} |\phi(ga_\e\cdot x_0)| \leq C \cdot |\log \e|, $$
see \cite{BR} Sect.1, Proposition part (3). 
\end{rem}

\section{Automorphic cusp forms}

In this section we explain how one can use the 
unipotent model for the crown domain in the 
theory of automorphic functions on the 
upper half plane. 

\par To avoid extra notation we will stick to  
$$\Gamma=\Sl(2, \Z)$$
for our choice of lattice. 

\par In the sequel we let $\phi$ be a Maa\ss{} cusp form. Let us fix $y>0$ and consider
the $1$-periodic function

$$F_y: \R \to \C, \ \ u\mapsto \phi(n_ua_y(i))=\phi(u+iy)\, .$$
This function being smooth and periodic admits a Fourier expansion

$$F_y(u)=\sum_{n\neq 0} A_n(y) e^{2\pi i n x}\, .$$
Here,  $A_n(y)$ are complex numbers depending on $y$.
Now observe that
$$n_ua_y=a_y a_y^{-1}n_u a_y= a_y n_{u/y}$$
and so
$$F_y(u)=\phi(a_y n_{u/ y}.x_0)\, .$$
As $\phi$ is a $\D(X)$-eigenfunction, it admits a
holomorphic continuation to $\Xi=X\times \oline X$. 
So we employ the crown model and conclude that 
$F_y$ admits a holomorphic continuation to the
strip domain
$$S_y=\{ w=u+iv \in \C\mid |v|< y\}\, .$$
Let now $\e>0$, $\e$ small. Then, for $n>0$, we
proceed with Cauchy

\begin{align*} A_n(y) & =\int_0^1 F_y( u -i(1-\e)y) e^{-2\pi i n (u-i(1-\e)y)}\ du \\
&= e^{-2\pi n (1-\e)y} \int_0^1 F_y( u-i(1-\e)y) e^{-2\pi i n u} \ du \\
&= e^{-2\pi n (1-\e)y} \int_0^1 \phi ( a_y n_{u/y} n _{-i(1-\e)}.x_0) e^{-2\pi i n u} \ du \, .
\end{align*}

Thus we get, for all $\e>0$ and $n\neq 0$ the inequality

\begin{equation}\label{eq=ineq}
|A_n(y)|\leq e^{-2\pi |n| y(1-\e)} \sup_{\Gamma g\in \Gamma\backslash G} |\phi(\Gamma g n_{\pm i(1-\e)}.x_0)|
\end{equation}

We need an estimate.

\begin{lem} \label{lem=esti}Let $\phi$ be a Maa\ss{} cusp form. Then there exists a
constant $C$ only depending on $\lambda$ such that
for all $0< \e < 1$
$$  \sup_{\Gamma g\in \Gamma\backslash G} |\phi(\Gamma g n_{i(1-\e)}.x_0)|\leq  C |\log \e|^{1\over 2}$$
\end{lem}

\begin{proof} Let $-\pi/4< t_\e <\pi/4$ be such that $\pm (1-\e)=\sin 2t_\e$.
Then, by Lemma \ref{l=match}  we have
$G n_{\pm i(1-\e)}.x_0 =G a_{\e}.x_0$ with
$a_\e=\begin{pmatrix} e^{it_\e} & 0 \\ 0& e^{-it_\e}\end{pmatrix}$.
Now note that $t_\e \approx \pi/4 -\sqrt{2\e}$ and thus
Prop. \ref{p=est}  and Theorem \ref{t=br} , give that
$$\sup_{\Gamma g\in \Gamma\backslash G} |\phi(ga_\e.x_0)|\leq C |\log \e|^{1\over 2}\, .$$
This concludes the proof of the lemma.\end{proof}

We use the estimates in Lemma \ref{lem=esti} in (\ref{eq=ineq}) and
get

\begin{equation}\label{eq=ineq1}
|A_n(y)|\leq C e^{-2\pi |n| y(1-\e)} |\log \e|^{1\over 2}\, ,
\end{equation}
and specializing to $\e=1/y$ gives that
\begin{equation}\label{eq=ineq2}
|A_n(y)|\leq C e^{-2\pi |n| (y-1)} (\log y)^{1\over 2} \, .
\end{equation}
This in turn yields for $y>2$ that

\begin{align*} |\phi(iy)| & = |F_y(0)|\leq \sum_{n\neq 0} |A_n(y)|\\
&\leq C (\log y)^{1\over 2}\sum_{n\neq 0} e^{-2\pi |n| (y-1)}\\
&\leq C (\log y)^{1\over 2} \cdot e^{-2\pi y}\end{align*}
It is clear, that we can replace $F_y$ by $F_y(\cdot +x)$ for any $x\in \R$
without altering the estimate. Thus
we have proved:

\begin{thm} Let $\phi$ be a Maa\ss{} cusp form. Then there
exists a constant $C>0$, only depending on $\lambda$,  such that
$$|\phi(x+iy)|\leq C (\log y)^{1\over 2} \cdot e^{-2\pi y} \qquad (y>2)\, .$$
\end{thm}

\begin{rem} It should be mentioned that this estimate is not optimal:
one can drop the $\log$-term by employing our knowledge about 
the coefficient functions  $A_n(y)$. However the method presented 
above generalizes to all semi-simple Lie groups.
\end{rem}

\section{$G$-innvariant Hilbert spaces of holomorphic functions on $\Xi$}

Hilbert spaces of holomorphic functions are in particular 
reproducing kernel Hilbert spaces, cf. {\cite{NA}. 
\subsection{General theory}
In this subsection $G$ is a group and $M$ 
is a second countable complex manifold.  The compact-open topology 
turns $\O(M)$ into a Fr\'echet space. 

\par We assume that $G$ acts on $M$ 
in a biholomorphic manner. This action 
induces an action of  $G$ on $\O(M)$ via:  

$$G\times \O(M)\to \O(M), \ \ (g, f)\mapsto f(g^{-1}\cdot)\, .$$ 
We assume that the action is continuous. 
By a $G$-invariant Hilbert space of holomorphic functions
on $M$ we understand a Hilbert space $\H\subset \O(M)$ such that 
\begin{itemize}
\item The inclusion $\H\hookrightarrow \O(M)$ is continuous;   
\item $G$ leaves $\H$ invariant and the action is unitary. 
\end{itemize}
 
It follows that all point evaluations 
$$\K_m: \H\to \C, \ \ f\mapsto f(m);
\quad (m\in M)$$
are continuous, i.e. $f(m)=\langle f, \K_m\rangle$.  
We obtain a kernel function 
$$\K: M\times M\to \C, \ \ (m,n)\mapsto \langle \K_n, \K_m\rangle= \K_n(m)$$
which is holomorphic in the first and anti-holomorphic 
in the second variable.  The kernel $\K$ characterizes $\H$ completely. 
Moreover that $G$ acts unitarily just means that $\K$ is $G$-invariant: 

$$\K=\K(g\cdot, g\cdot)\qquad (g\in G)\, .$$ 

We denote by $\cC=\cC(M,G)$ the cone of all $G$-invariant holomorphic  
positive definite kernels (i.e reproducing kernels) on 
$M\times \oline M$. In the terminology of Thomas \cite{Tho}  
is a conuclear cone in the Fr\'echet space $\O(M\times \oline M)$ 
and as such admits a decomposition 

\begin{equation} \label{r1} 
\K=\int_{\Ext(\cC)} \K^\lambda \ d\mu(\lambda)\,, \end{equation}
 see \cite{k1}, Th. II.12  for a more general statement. 
In (\ref{r1}) the symbol  $\Ext(\cC)$ denotes the  equivalence classes 
(under $\R^+$-scaling) of extremal rays in $\cC$ and 
$$\lambda\mapsto \K^\lambda$$
is an appropriate assignment of representatives; 
furthermore $\mu$ is a Borel measure on $\Ext(\cC)$.

\subsection{Invariant Hilbert spaces on the crown}
We return to $G=\Sl(2,\R)$ and $M=\Xi$. We write 
$\hat G_{\rm sph}$ for the $K$-spherical part of 
$\hat G$ and note that the map 
$\lambda \mapsto [\pi_\lambda]$
is a bijection from   $\left(\R\cup (-i,i)\right)/\W$
to $\hat G_{\rm sph}$. 
Morover for $[\pi]\in\hat G_{\rm sph}$ we define a positive definite
holomorphic $G$-imvariant kernel $\K^\pi$ on $\Xi$ via

$$\K^\pi(z,w)=\langle \pi(z)v, \pi(w)v\rangle \qquad (z,w\in\Xi)$$
where $v$ is a unit $K$-fixed vector. 
Then each kernel $\K$ of a $G$-invariant Hilbert space $\H\subset 
\O(\Xi)$ can be written as 

\begin{equation}\label{e=1} 
\K(z,w)=\int_{\hat G_{\rm sph}} \K^\lambda(z,w) \ d\mu(\lambda)
\qquad(z,w\in\Xi)\end{equation}
where we simplified notation $\K^{\pi_\lambda}$ to $\K^\lambda$. 
The Borel measure $\mu$ satisfies the condition 

\begin{equation}\label{e=2}(\forall 0<c<2)\qquad 
\int_{\hat G_{\rm sph}} e^{c\left|\Re \lambda\right|}\ d\mu(\lambda)<\infty\, 
\end{equation}
and conversely, a measure $\mu$ which satisfies (\ref{e=2}) 
gives rise to a $G$-invariant Hilbert space of holomorphic functions 
on $\Xi$, see \cite{KSII}, Prop. 5.4. 

\subsection{Hardy spaces for the most continuous spectrum of the hyperboloid}
A little bit of motivation upfront. 
We recall the splitting of square integrable functions on $\R$  
$$L^2(\R)\simeq H^2(X)\oplus H^2(\oline X)$$
into a sum of Hardy spaces:
$$H^2(X)=\left \{ f\in\O(X)\mid \sup_{y>0} \int_\R |f(x+iy)|^2 \ dx <\infty\right\}\, ,$$
and 
$$H^2(\oline X)=\left \{ f\in\O(\oline X)\mid \sup_{y<0} 
\int_\R |f(x+iy)|^2 \ dx <\infty\right\}\, .$$

The isomorphism map from $H^2(X)$ to $L^2(\R)$ is just 
the boundary value: 
$$b:H^2(X)\to L^2(\R); \ b(f)(x)=\lim_{y\to 0^+} f(\cdot +iy)$$
and likewise 
$$\oline b:H^2(X)\to L^2(\R); \ \oline b(f)(x)=\lim_{y\to 0^-} f(\cdot +iy)$$

In the sequel we replace the pair $(\R, X)$, i.e. Shilov boundary $\R$ 
of the complex manifold $X$, by $(Y, \Xi)$ where $Y=G/H=\Sl(2,\R)/\SO(1,1)$
is the distinguished boundary of $\Xi$.  
But now we have to be more careful with the space of square-integrable functions
$L^2(Y)$. Recall the Casimir element, the generator of ${\mathcal Z}(\gf):=\U(\gf)^G$: 
$${\bf C}= {\bf h}^2 + 4 {\bf e}{\bf f}\, .$$
Then 
$$L^2(Y)=L^2(Y)_{\rm mc}\oplus L^2(Y)_{\rm disc}$$ 
accordingly whether ${\bf C}$ has continuous or discrete spectrum.
Here our concern is only with the (most) continuous part
$L^2(Y)_{\rm mc}$. So it is about to define 
the Hardy spaces $H^2(\Xi)$ and $H^2(\oline \Xi)$. 
It was a result of \cite{GKO} that $H^2(\Xi)$ actually exists and that 
the kernel is given (up to positive scale) by 

\begin{equation} \label {h=1} \K = \int_0^\infty \K^\lambda 
\ { \lambda \tanh (\pi \lambda/2)  \over
\cosh (\pi\lambda)}\cdot   d\lambda
\, .\end{equation}
There exists a well defined boundary value map 
$$b: H^2(\Xi)\to L_{\rm mc}^2(Y);\  b(f)(g(1,-1))=\lim_{e\to 0^+} f(ga_\e\cdot 
x_0)$$
which is equivariant and isometric. 
Likewise one has a Hardy space $H^2(\oline \Xi)$ on $\oline \Xi$ which is 
just the complex conjugate of $H^2(\Xi)$. The decomposition of the 
continuous spectrum then is \cite{GKO}: 

\begin{equation} L_{\rm mc}^2(Y)=b(H^2(\Xi))\oplus \oline{b} (H^2(\oline \Xi))
\end{equation}
 
\begin{rem} (a) We caution the reader that $L_{\rm mc}^2(Y)$ is not exhausted 
by  our Hardy spaces once the real rank of $Y$ is larger then one.
\par\noindent  (b) We defined $H^2(\Xi)$ by its kernel and not by its norm.
It is possible to give a geometric expression of the norm in $H^2(\Xi)$
in terms of certain $G$-orbital integrals on $\Xi$, see \cite{GKO2}. 
This method was also quite useful  in our work on the 
heat kernel transform \cite{KOS}. 
\end{rem} 
 
\section{K\"ahler structures on $\hat G_{\rm sph}$}

Throughout this section $(\pi,\H)$ denotes a non-trivial 
$K$-spherical unitary representation of $G$. 
We let $v_K\in \H$ be a $K$-fixed unit vector.

\par We first recall that the projective  space 

$$\P(\H) =\H^\times / \C^*$$
of $\H$ is an infinite dimensional complex manifold
which is complete under the Fubini-Study metric $g_{FS}$. 
We write 

$$h_{FS}=g_{FS} + i \omega_{FS}$$
for the corresponding Hermitian structure on $\P(\H)$. 

Without proof we  state two results, see \cite{KSII}, Prop. 3.1
and Th. 3.3, which hold in full generality. 

\begin{thm} The map 
$$F_\pi: \Xi\to \P(\H), \ \ z\mapsto [\pi(z)v_K]$$
is proper.  In particular ${\operatorname im} F_\pi$  is closed and the pull back 
$h_\pi:=F_\pi^*h_{FS}$ defines a Hermitian K\"ahler  structure on $\Xi$
whose underlying Riemannian structure $g_\pi$ is complete. 
\end{thm}

\begin{rem} Elementary complex analysis shows that the map 
$$s_\pi: \Xi\to \R_{>0}, \ \ z\mapsto \|\pi(z)v\|^2$$
is strictly plurisubharmonic. The K\"ahler form 
$\omega_\pi$ from the 
previous theorem is then nothing else as 
$$\omega_\pi(z)={i\over 2}\partial\oline \partial 
\log \|\pi(\cdot )v_K\|^2\, .$$
\end{rem}

The main result of this section then is, see \cite{KSII}, : 

\begin{thm} The map $\pi\mapsto \omega_\pi$ identifies 
$\hat G_{\rm sph}\setminus\{ {\bf 1}\}$ with positive K\"ahler forms 
on $\Xi$ whose associated Riemannian metric is complete. 
\end{thm}  

The big problem then is to characterize the image of $\pi\mapsto \omega_\pi$.

\section{Appendix: The hyperbolic model of the crown domain}

The upper half plane $X=G/K$ does not depend on the 
isogeny class of $G$. Replacing $G$ by its adjoint 
group $\operatorname{PSl}(2,\R)\simeq \rm{SO}_e(1,2)$ 
has essentially no consequences for the crown. 
Changing the perspective to $G=\rm{SO}_e(1,2)$ we obtain 
new view-points by realizing $\Xi$ in the complex quadric. 
This is the topic of this section.

\par Let us fix the notation first. From now on 
 $G=\rm{SO}_e(1,2)$ and we regard $K={\rm SO}(2,\R)$ as a maximal
compact subgroup of $G$ under the standard lower right
corner embedding.
\par Let us define a quadratic form $Q$ on $\C^3$ by

$$Q ({\bf z})=z_0^2-z_1^2 -z_2^2, \qquad
{\bf z}=(z_0, z_1, z_2)^T\in \C^3\, .$$
With $Q$ we declare real and complex hyperboloids
by
$$X=\{ {\bf x}=(x_0, x_1, x_2)^T\in \R^3\mid Q({\bf x})=1, x_0>0\}$$
and
$$X_\C=\{ {\bf z} =(z_0,z_1,z_2)^T\in \C^3\mid
Q({\bf z})=1\}\ .$$
We notice that  mapping
$$G_\C/K_\C\to X_\C, \ \ gK_\C\mapsto g.{\bf x}_0 \qquad ({\bf x}_0=(1,0,0))$$
is diffeomorphic and that $X$ is identified
with $G/K$.
\par
At this point it is useful to introduce coordinates
on $\gf=\so(1,2)$. We set

\begin{equation*}
{\bf e_1}=\begin{pmatrix} 0 & 0 & 1 \\ 0 & 0 & 0\\ 1 & 0 & 0\end{pmatrix},
\quad
{\bf e_2}=\begin{pmatrix} 0 & 1 & 0 \\ 1 & 0 & 0\\ 0 & 0 & 0\end{pmatrix},
\quad
{\bf e_3}=
\begin{pmatrix} 0 & 0 & 0 \\ 0 & 0 & 1\\ 0 & -1 & 0\end{pmatrix}\, .
\end{equation*}
We notice that $\kf=\R {\bf e_3}$, $\pf= \R{\bf e_1} \oplus  \R {\bf e_2}$
and make our choice of the flat piece $\af=\R{\bf e_1}$.
Then $\Omega=(-1, 1) {\bf e_1}$,
$\Xi=G\exp (i(-\pi/ 2,\pi/2){\bf e_1}).{\bf x_0}$
and we obtain Gindikin's favorite model of the
crown
$$\Xi=\{ {\bf z}={\bf x} +i{\bf y}\in X_\C\mid
x_0>0, Q({\bf x}) >0\}\, .$$
It follows that the boundary
of $\Xi$ is given by

\begin{equation} \label{be1}
\partial \Xi=\partial_s \Xi \amalg \partial_n \Xi\end{equation}
with semisimple part
\begin{equation} \label{be2}
\partial_s \Xi=\{ i{\bf y}\in i\R^3\mid
Q({\bf y})=  -1\}
\end{equation}
and nilpotent part
\begin{equation} \label{be3}
\partial_n \Xi=\{  {\bf z}={\bf x}+i{\bf y}
\in X_\C\mid x_0>0, Q{(\bf x})=0\}\, .
\end{equation}

Notice that
${\bf z_1}=\exp(i\pi/2 {\bf e_1}).{\bf x_0}=(0,0,i)^T$
and that the stabilizer of ${\bf z}_1$ in $G$ is
the symmetric subgroup $H={\rm SO}_e(1,1)$, sitting inside of $G$
as the upper left corner block.
Hence
\begin{equation}
\partial_s\Xi=\partial_d \Xi =G.{\bf z_1}\simeq G/H
\end{equation}

\par A first advantage of the hyperbolic model is a 
more explicit view on the boundary of $\Xi$: Proposition 
\ref{p=b} becomes more natural in these coordinates. 
We allow ourselves to go over this topic again. 

\par Write $\tau$ for the involution on $G$
with fixed point set $H$ and let
$\gf=\h \oplus \qf$ the corresponding $\tau$-eigenspace
decomposition. Clearly, $\h=\R {\bf e_2}$ and
$\qf=\af \oplus \kf=\R {\bf e_1} \oplus  \R {\bf e_3}$.
Notice that $\qf$ breaks as an $\h$-module into two pieces
$$\qf=\qf^+ \oplus  \qf^-$$
with
$$\qf^\pm=\{ Y\in \qf\mid [e_2, Y]=\pm Y\} =\R ({\bf e_1}\pm {\bf e_2})
\, .$$
Let us define the $H$-stable pair of half lines
$$C =\R_{\geq 0}({\bf e_1}\oplus  {\bf e_3})\cup  \R_{\geq 0}
({\bf e_1}- {\bf e_3})$$
in $\qf=\qf^+ \oplus \qf^-$.  We remark
that $C$ is the boundary of the $H$-invariant
open cone
$$W=\Ad(H)(\R_{> 0}{\bf e_1})= \R_{> 0}({\bf e_1}+ {\bf e_3})\oplus  \R_{> 0}
({\bf e_1}- {\bf e_3})\, .$$
Recall that the tangent bundle $T(G/H)$ naturally identifies
with $G\times_H \qf$ and let us mention that $\cC=G\times_H C$ is a $G$-invariant
subset thereof.  Proposition \ref{p=b} from before now reads 
as:  

\begin{prop}\label{p=b'} For $G={\rm SO}_e(1,2)$, the mapping
$$b: G\times_H C\to \partial \Xi, \ \ [g,Y]\mapsto g\exp(-iY).{\bf z_1}$$
is a $G$-equivariant homeomorphism.
\end{prop}

\par As a second application of the hyperbolic 
model we now prove the orbit-matching Lemma {l=match} from 
before. 

\bigskip\noindent {\bf Proof of Lemma \ref{l=match}}. With $\af=\R {\bf e_1}$ we come to our choice of $\nf$.
For $z\in \C$ let

$$n_z=\left( \begin{array}{ccc} 1+\frac{1}{2} z^2 & z &
-{1\over 2}   z^2  \\ z & 1 & -z \\
\frac{1}{2} z^2  & z & 1-\frac{1}{2}
z^2 \end{array} \right)
$$
and
$$N_\C=\{ n_z \mid z\in\C\}\, .$$

Further for $t\in \R $ with $|t|<{\pi\over 2}$ we set

$$a_t=\left( \begin{array}{ccc} \cos t & 0 & -i\sin t \\ 0 &
1 & 0 \\ -i \sin t & 0 & \cos t \end{array} \right) \in
\exp (i\Omega)\, .$$
\par The statement of the lemma translates into the
assertion
\begin{equation} Gn_{i\sin t }.{\bf x}_0=G a_t.{\bf x}_0\, .\end{equation}

Clearly, it suffices to prove that
$$a_t.{\bf x}_0=(\cos t, 0, -i\sin t)^T\in Gn_{i \sin t}.{\bf x}_0\, .$$
Now let $k\in   K$ and $b\in A$ be elements
which we write as

$$k= \left( \begin{array}{ccc} 1 & 0 & 0 \\ 0 &
\cos \theta & \sin\theta \\ 0 & -\sin \theta  & \cos \theta
\end{array} \right) \quad \hbox{and}\quad
b=\left( \begin{array}{ccc} \cosh r & 0 & \sinh r  \\ 0 &
1& 0 \\ \sinh r  & 0 & \cosh r \end{array} \right)$$
for real numbers $r,\theta$.
For $y\in \R$, a simple computation yields that

$$kbn_{iy}.{\bf x}_0= \left(\begin{array}{c}
\cosh r (1-{1\over 2} y^2) - {1\over 2}
y^2\sinh r \\
 iy\cos\theta + \sin\theta( \sinh r (1-{1\over 2} y^2) -
{1\over 2}y^2
\cosh r)\\ -iy\sin\theta +
\cos\theta( \sinh r (1-{1\over 2}y^2) -{1\over 2}y^2
\cosh r)
\end{array} \right)\, . $$
Now we make the choice of $\theta={\pi \over 2}$ which gives us
that

$$kb n_{iy}.{\bf x}_0= \left(\begin{array}{c}
\cosh r (1-{1\over 2} y^2) - {1\over 2}
y^2\sinh r \\
 \sinh r (1-{1\over 2}y^2) -{1\over 2}y^2
\cosh r\\ -iy
\end{array} \right)\, . $$
As $y=\sin t $ we only have to verify that we can choose $r$ such that
 $\sinh r (1-{1\over 2}y^2) -{1\over 2}y^2
\cosh r=0$. But this is equivalent to
$$\tanh r =  {{1\over 2} y^2 \over 1 -{1\over 2}y^2}\, .$$
In view of $-1< y=\sin t <1 $, the right hand side is smaller than one
and we can solve for $r$.
\qed

\end{document}